\documentclass[pdflatex,sn-mathphys-num]{sn-jnl}

\usepackage[table]{xcolor}

\usepackage{comment}
\usepackage{graphicx}%
\usepackage{amsmath,amssymb,amsfonts}%
\usepackage{amsthm}%
\usepackage{textcomp}%
\usepackage{mathtools}
\usepackage[T1]{fontenc}

\theoremstyle{thmstyleone}%
\newtheorem{theorem}{Theorem}[section]

\newtheorem{proposition}[theorem]{Proposition}
\newtheorem{lemma}[theorem]{Lemma}
\newtheorem{corollary}[theorem]{Corollary}

\theoremstyle{thmstyletwo}%
\newtheorem{example}[theorem]{Example}
\newtheorem{remark}[theorem]{Remark}

\theoremstyle{thmstylethree}%
\newtheorem{definition}[theorem]{Definition}

\newcommand{\PosId}{\operatorname{Id}^{+}}
\newcommand{\Sk}{\operatorname{Sk}}
\newcommand{\tprod}[2]{#1\ast_{\tau}#2}
\newcommand{\CCtau}{\operatorname{CC}_{\tau}}
\newcommand{\CCm}{\operatorname{CC}_{\mathrm m}}
\newcommand{\layer}[1]{L_{#1}}
\newcommand{\Ctau}{\mathcal C_{\tau}}
\newcommand{\Cm}{\mathcal C_{\mathrm m}}
\newcommand{\Ltau}{\mathcal L_{\tau}}
\newcommand{\Lm}{\mathcal L_{\mathrm m}}

\makeatletter
\newcommand{\taggeditem}[2]{%
 \item[#1]%
 \def\@currentlabel{#1}%
 \label{#2}%
}
\makeatother

\begin{document}

\title[Canonical direct-system representation and rigidity of local-unit-aligned monoids]{A canonical rigid direct-system representation of finite local-unit-aligned totally ordered monoids\footnote{This version is identical to the version in \emph{Semigroup Forum}.}
}

\author*[1,2]{\fnm{S\'andor} \sur{Jenei}}\email{jenei.sandor@uni-eszterhazy.hu, jenei@ttk.pte.hu}

\affil*[1]{\orgdiv{Institute of Mathematics and Informatics}, \orgname{Eszterh\'azy K\'aroly Catholic University}, \orgaddress{\country{Hungary}}}
\affil[2]{\orgdiv{Institute of Mathematics and Informatics}, \orgname{University of P\'ecs}, \orgaddress{\country{Hungary}}}

\affil{\\\medskip Communicated by Mahir Bilen Can}

\affil{{\small Received 18 May 2026 / Accepted 05 August 2026}}

\abstract{We study finite local-unit-aligned totally ordered monoids, that is, finite totally ordered monoids in which each element has coinciding greatest right and left local units.
We prove that every such monoid admits a canonical rigid chain-indexed direct-system representation, and conversely that every rigid system of the corresponding kind reconstructs a finite local-unit-aligned totally ordered monoid.
The representation is induced intrinsically by the local-unit map \(\tau\), through the canonical stratification of the positive idempotent skeleton into \(\tau\)-multiplication-coherent blocks.
More precisely, from \(\tau\) we construct component monoids and transition maps forming a strictly compatible finite chain-indexed direct system from which both the ambient order and the ambient multiplication are recovered.
In the finite case, strict compatibility forces every proper transition map to be unit-constant; this rigidity makes the canonical components \(\tau\)-multiplication-cohesive and yields a Clifford-type ordinal-sum-like reconstruction theorem for finite local-unit-aligned totally ordered monoids.}



\maketitle

\section{Introduction}\label{IntroONE}

Finite totally ordered monoids provide a natural setting in which algebraic and order-theoretic structure are forced to interact.
In this paper we study the \emph{local-unit-aligned} case, that is, finite totally ordered monoids in which for every element the greatest right and left local units coincide.
Our aim is to show that this class admits a canonical representation by direct systems over finite chains.
A noteworthy feature of the construction is a finite rigidity phenomenon.
Although the componentwise representation is initially obtained as a direct system of local-unit-aligned monoids, the compatibility maps are far more rigid than arbitrary connecting maps.
In the finite case, strict compatibility forces every connecting map to be constant at the unit of the target component.
Thus the resulting representation is not merely a direct-system decomposition, but collapses to an ordinal-sum-like assembly in which higher components are seen from lower ones only through their local units.

\medskip
\noindent\textbf{Main theorem, informally.}
Every finite local-unit-aligned totally ordered monoid admits a canonical rigid finite chain-indexed direct-system representation by \(\tau\)-multiplication-cohesive component monoids, and conversely every such rigid system reconstructs such a monoid.
\medskip

The starting point is the local-unit map \(\tau\colon M\to \PosId(\mathbf M)\), which sends each element to its local unit in the positive idempotent skeleton.
The guiding principle of the paper is that this map induces an intrinsic stratification of the monoid, and that this stratification leads to canonical components and transition maps.
In the finite setting considered here, the resulting direct system is rigid in the sense that its proper transition maps are forced to collapse to the target block unit.

To make this precise, we define two canonical partitions of the positive idempotent skeleton.
The first is the \(\tau\)-cohesive partition, namely the finest $\tau$-stable partition.
The second is the \(\tau\)-multiplication-coherent partition, namely the finest partition for which the block of \(\tau(xy)\) depends only on the blocks of \(\tau(x)\) and \(\tau(y)\).
We show that the former always refines the latter, and that each block of a \(\tau\)-stable partition carries a canonical component monoid which is again finite, totally ordered, and local-unit-aligned.

The first structural result states that the component monoids attached to the \(\tau\)-multiplication-coherent partition form a strictly compatible direct system over a finite chain, and that this system reconstructs the original monoid completely.
Conversely, every strictly compatible finite chain-indexed direct system of finite local-unit-aligned totally ordered monoids assembles to such a monoid.
In this way, intrinsic decomposition and external construction meet in a single representation principle.

A second theme of the paper is rigidity.
We prove that, in the finite totally ordered setting, strict compatibility forces every proper transition map to be unit-constant; this yields a rigidity theorem for the canonical transition maps.
As a consequence, the canonical component monoids are already \(\tau\)-multiplication-cohesive.
More generally, for rigid finite chain-indexed direct systems the canonical $\tau$-multiplication-coherent decomposition of the resulting monoid refines the given construction-level partition, and the two partitions coincide when all constituent monoids are \(\tau\)-multiplication-cohesive.
Thus the $\tau$-stratification is not merely descriptive, but structurally definitive.

Ordered monoids and ordered semigroups occur in several parts of algebra, including ordered algebra, semigroup theory, valuation theory, and algebraic logic; see, for instance, \cite{Fuchs1963}.
In the present paper we work in a finite setting, but we do not impose either of the two usual conventions from the tomonoid literature, namely that the identity be the bottom element (the ``positive'' convention) or the top element (the ``negative'' convention).
Instead, the identity may lie anywhere in the chain.
The framework therefore contains both conventions as special cases and does not assume commutativity.

The paper belongs to a broader representation-and-construction tradition.
Classical work on naturally or positively ordered commutative semigroups analyzed Archimedean and nilpotent structure as well as ordinal-sum type decomposition phenomena \cite{Clifford1954,Clifford1958,Saito1968,Saito1976I,Saito1976II,Saito1976III, Satyanarayana1979Arch,KehayopuluTsingelis2006,Gabovich1976,Vetterlein2015}, and extension-theoretic aspects of ordered semigroups were studied via ideal extensions \cite{Hulin1976,KehayopuluTsingelis2003}.
In the commutative, especially finite commutative, setting, related structures have been studied extensively under the name of \emph{tomonoids}; see \cite{Whipple2005,EvansEtAl2001,Horcik2010,Vetterlein2016Pos,Vetterlein2017Rep, Vetterlein2015,PetrikVetterlein2014,PetrikVetterlein2016,PetrikVetterlein2017, PetrikVetterlein2019,Vetterlein2016Real}.
In the integral commutative residuated setting, closely related ordered-monoidal reducts occur in linearly ordered residuated algebras, especially BL- and MTL-chains, where decomposition phenomena have also been studied from several complementary viewpoints \cite{AglianoMontagna2003,Busaniche2005,MontagnaNogueraHorcik2006, HorcikMontagna2009,Horcik2011,CastiglioniZuluaga2021,Horcik2007}.

The present approach is different in that it starts from an intrinsic monoid-theoretic mechanism, namely the local-unit geometry encoded by \(\tau\).
The idea of stratifying a structure by means of the local-unit map \(\tau\), in particular through the layers $ M_u=\{x\in M:\tau(x)=u\}, $ was introduced first in the residuated setting in \cite{Jenei2022GroupRepr}.
In that earlier framework, for those very layers \(M_u\), the local-unit map satisfied the strong global identity $ \tau(xy)=\tau(x)\tau(y) $ equivalently, in the totally ordered positive idempotent skeleton, $ \tau(xy)=\max\{\tau(x),\tau(y)\}. $
But this phenomenon is genuinely exceptional: such a global multiplicativity law for $\tau$ is by no means automatic even in residuated chains, and it is still less to be expected in the wider setting of totally ordered monoids.
The novelty of the present paper is that the representation theory is shown not to depend on that exceptional identity.
Instead, the relevant multiplicative content of \(\tau\) is recovered at block level, by passing from the exact layers to the finest \(\tau\)-multiplication-coherent decomposition, equivalently, to the corresponding finest \(\tau\)-multiplication-coherent partition of the positive idempotent skeleton.
The resulting canonical partition, together with the induced canonical component monoids and transition maps, isolates precisely those intrinsic strata on which the interaction of \(\tau\) with multiplication is coherent enough to support the direct-system representation.
Thus the point is not merely to prove another decomposition theorem, but to identify the genuinely monoid-theoretic mechanism that survives after the exceptional global law has disappeared.
In this sense, what remains from the earlier residuated picture is not the global identity itself, but the canonical structural decomposition that replaces it.

The representation obtained here is not an externally imposed construction or the result of an additional condition on the monoid, but an intrinsic decomposition forced by the monoid's own local-unit geometry.
Its novelty is not merely that the decomposition is canonical.
Archimedean decompositions, for example, are canonical once the relevant Archimedean equivalence relation has been fixed.
The contribution is that this canonical stratification arises intrinsically from the local-unit map \(\tau\): instead of stratifying elements by order-growth or mutual domination of powers, the construction stratifies the positive idempotent skeleton by the multiplicative behavior of \(\tau\).
The resulting quotient is the finest one for which the block of \(\tau(xy)\) is determined by the blocks of \(\tau(x)\) and \(\tau(y)\).
Thus the canonical object is not an Archimedean quotient, but the finest \(\tau\)-multiplication-coherent quotient of the positive idempotent skeleton.

This distinction is structural rather than terminological.
The resulting blocks are precisely those from which one can build component monoids and canonical transition maps that reconstruct the original order and multiplication.
Mixed products are therefore not governed by an externally prescribed ordinal-sum rule, by ideal-extension data, or by Archimedean comparability alone.
Rather, the original multiplication determines canonical transition maps between the \(\tau\)-coherent components, and these maps are the data through which the representation encodes all cross-component products.
The finite rigidity theorem then shows that the proper transition maps are forced to be unit-constant, so cross-component multiplication collapses to absorption by the upper component's local unit as a theorem rather than as part of the construction.

There are two distinct Clifford precedents behind the present representation theorem.
The first is Clifford's theorem that Clifford inverse semigroups, equivalently, inverse semigroups with central idempotents, are precisely strong semilattices of groups \cite{Clifford1941,CliffordPreston1961}.
Here the idempotent semilattice indexes the maximal subgroups, and coherent connecting homomorphisms govern products between different components.
P\l{}onka later developed the corresponding component-and-transition-map construction for arbitrary algebras \cite{Plonka1967}.
The analogy here is structural rather than literal: our objects are finite totally ordered monoids, not inverse semigroups, and no inverse operation is assumed.
Still, the same representation-theoretic idea is present: a global algebraic object is recovered from local components and canonical connecting maps.
The second, and more order-specific, precedent is Clifford's ordinal-sum theory for naturally totally ordered commutative semigroups \cite{Clifford1954,Clifford1958}.
There the linearly ordered structure is built from ordered components arranged in an ordinal-sum-like way, with mixed products controlled by absorption between different levels.
This is the closer analogy for the rigidity part of the present paper.
Before rigidity is used, our canonical decomposition is genuinely a direct-system representation: elements from different canonical components interact through transition maps.
The finite rigidity theorem then forces every proper canonical transition map to be unit-constant.
Consequently, whenever one component lies above another, mixed products no longer depend on a nontrivial transition image from the lower component; the higher component absorbs the lower one.
Thus the rigid representation collapses to an ordinal-sum-like assembly, in the precise sense that cross-component multiplication is governed by the upper component's local unit.
Crucially, this unit-constancy is not part of Clifford's strong-semilattice theorem, whose connecting homomorphisms may be nontrivial; it is an additional conclusion forced here by strict compatibility, total order, and finiteness.

The present theory therefore has two Clifford-type aspects.
It is Clifford-like in the strong-semilattice sense because it reconstructs a global monoid from components and canonical maps.
It is also Clifford-like in the ordinal-sum sense because, once finiteness forces rigidity, those canonical maps become degenerate enough that the direct-system representation behaves like an ordered component sum with absorption across levels.
In this sense, the present representation theorem shows that Clifford-style reconstruction continues to operate even in an ordered-monoid setting without inverses, and in fact does so simultaneously at the algebraic and order-theoretic levels.
The novelty is that both phenomena arise intrinsically from the local-unit map \(\tau\), in a finite, possibly noncommutative setting where the order need not be natural in Clifford's sense and the identity need not lie at either endpoint of the chain.

The order-theoretic side of the reconstruction is supplied by the \emph{directed lexicographic order}, introduced in \cite{Jenei2022GroupRepr}.
For a chain-indexed direct system, this order compares two elements by transporting them to a common upper component, comparing them there, and resolving equality according to the direction of the index chain.
In the present context it is the order-theoretic counterpart of the direct-system reconstruction of multiplication.

The paper develops in the following logical order, with the canonical rigid representation theorem as its destination.
We begin by extracting from the local-unit map \(\tau\) the \(\tau\)-cohesive partition and the \(\tau\)-multiplication-coherent partition of the positive idempotent skeleton of the monoid and by showing that their blocks carry natural component monoid structures.
On this basis, we construct the associated canonical direct system and prove that it reconstructs the original ordered monoid structure.
We then establish the converse direction, showing how suitable direct systems of totally ordered monoids give rise to finite local-unit-aligned totally ordered monoids.
After this representation-theoretic core is in place, we prove the rigidity result, which forces strictly compatible transition maps to be unit-constant.

The paper ends with the canonical rigid representation theorem and its reconstruction consequences.

\section{Preliminaries}

By a \emph{chain} we mean a nonempty totally ordered set.
In particular, every finite chain has a least and a greatest element.
Throughout the paper, \(\mathbf M=\langle M,\le,\cdot,e\rangle\) denotes a \emph{finite totally ordered monoid}: thus $\langle M,\le\rangle$ is a finite chain, $\langle M,\cdot,e\rangle$ is a monoid, and multiplication is isotone in both arguments, that is,
\[
 x\le y \Longrightarrow xz\le yz\quad\text{and}\quad zx\le zy
 \qquad(z\in M).
\]

\begin{definition}
Let $\mathbf M$ be a finite totally ordered monoid.

\begin{enumerate}
\item An element $p\in M$ is \emph{positive} if $ p\ge e. $

\item An element $p$ is \emph{idempotent} if $p^{2}=p$.
We write $\PosId(\mathbf M)$ for the set of positive idempotents of $\mathbf M$.

\item For $x\in M$, define \(R(x):=\{z\in M:xz\le x\}\) and \(L(x):=\{z\in M:zx\le x\}.\)
Since $e\in R(x)\cap L(x)$ and the chain is finite, both sets have greatest elements greater than or equal to $e$.
We write \(\tau_r(x):=\max R(x)\) and \(\tau_{\ell}(x):=\max L(x).\)
\item The monoid $\mathbf M$ is called \emph{local-unit-aligned} if for every $x\in M$ the two maxima coincide, \(\tau_r(x)=\tau_{\ell}(x).\)
In this case we write \(\tau(x):=\tau_r(x)=\tau_{\ell}(x)\) and call $\tau$ the \emph{local-unit map} of $\mathbf M$.

\item The monoid $\mathbf M$ is called \emph{integral} if $e$ is the largest element of $M$.
Integral monoids are evidently local-unit-aligned, as $\tau_r(x)=\tau_{\ell}(x)=e$.

\end{enumerate}
\end{definition}

\begin{remark}[Standing convention for \(\tau\)-constructions]
\label{rem:standing-tau-convention}
From this point on, any unqualified use, for a finite totally ordered monoid \(\mathbf N\), of the symbols
\[
\tau_{\mathbf N},\quad \tau,\quad \layer{\cdot},\quad
\tprod{\cdot}{\cdot},\quad
\Ctau(\mathbf N),\quad \Ltau(\mathbf N),\quad
\Cm(\mathbf N),\quad \Lm(\mathbf N)
\]
carries the standing assumption that \(\mathbf N\) is local-unit-aligned.
When several monoids are present, this convention applies separately to each monoid for which these symbols are used.
\end{remark}

The two examples in Fig.~\ref{fig:smallest-examples} are included for orientation: the first shows that local-unit-aligned totally ordered monoids need not be commutative, while the second shows that they need not be integral.
The asserted minimality is computational: an exhaustive enumeration of all finite totally ordered monoids on smaller chains, followed by a direct test of local-unit alignment and of the displayed distinguishing property, yields no example.

\begin{figure}[ht]
\centering
\begin{minipage}{0.45\textwidth}
\[
\mathbf M_1=
\left(
\begin{array}{cccc}
1 & 1 & 1 & 1 \\
1 & 1 & 1 & 2 \\
1 & 2 & 3 & 3 \\
1 & 2 & 3 & 4
\end{array}
\right)
\ \ 
\left(
\begin{array}{c}
4\\
4\\
4\\
4
\end{array}
\right)
\]
\end{minipage}
\hfill
\begin{minipage}{0.45\textwidth}
\[
\mathbf M_2=
\left(
\begin{array}{ccccc}
1 & 1 & 1 & 1 & 1 \\
1 & 1 & 1 & 2 & 2 \\
1 & 2 & 3 & 3 & 3 \\
1 & 2 & 3 & 4 & 5 \\
1 & 2 & 3 & 5 & 5
\end{array}
\right)
\ \ 
\left(
\begin{array}{c}
5\\
5\\
5\\
4\\
5
\end{array}
\right)
\]
\end{minipage}

\caption{Multiplication tables of the smallest noncommutative local-unit-aligned totally ordered monoid $\mathbf M_1$ (left) and of the smallest non-integral one $\mathbf M_2$ (right).
The accompanying column gives $\tau(i)$ for each $i$.}
\label{fig:smallest-examples}
\end{figure}

\begin{definition}
The totally ordered set
\[
 \Sk(\mathbf M):=\langle \PosId(\mathbf M),\le\!\upharpoonright_{\PosId(\mathbf M)}\rangle
\]
is called the \emph{skeleton} of $\mathbf M$.
\end{definition}

\begin{lemma}\label{lem:basic_tau}
Let $\mathbf M$ be local-unit-aligned and let $x\in M$.
Then:
\begin{enumerate}
\item\label{item:positive} $\tau(x)\ge e$;
\item\label{item:local-unit-equalities} $x\tau(x)=x=\tau(x)x$;
\item\label{item:tau-idempotent} $\tau(x)$ is a positive idempotent;
\item\label{item:tau-fixes-idempotents} if $u\in\PosId(\mathbf M)$, then $\tau(u)=u$;
\item\label{item:below-tau-gives-unit} if $u\in\PosId(\mathbf M)$ and $u\le \tau(x)$, then \(xu=x=ux.\)
\end{enumerate}
\end{lemma}

\begin{proof}
Since $e$ is the global unit of $\mathbf M$, $e\in R(x)\cap L(x)$, hence \eqref{item:positive} holds.

Since $\tau(x)\in R(x)\cap L(x)$, we have $x\tau(x)\le x$ and $\tau(x)x\le x$.
Because $\tau(x)$ is positive, also $x\le x\tau(x)$ and $x\le \tau(x)x$.
Hence $x\tau(x)=x=\tau(x)x$, proving \eqref{item:local-unit-equalities}.

For idempotence, we have $x\tau(x)^2=(x\tau(x))\tau(x)=x\tau(x)=x$, so $\tau(x)^2\in R(x)$ and therefore $\tau(x)^2\le \tau(x)$.
Since $\tau(x)$ is positive, we also have $\tau(x)\le \tau(x)^2$.
Thus $\tau(x)^2=\tau(x)$, proving \eqref{item:tau-idempotent}.

Now let $u\in\PosId(\mathbf M)$.
Since $uu=u$, we have $u\in R(u)$ and $u\in L(u)$, so $u\le \tau(u)$.
On the other hand, by \eqref{item:local-unit-equalities} applied to $u$, $u\tau(u)=u.$
Because $u$ is positive, $\tau(u)\le u\tau(u)=u$.
Hence $\tau(u)=u$, proving \eqref{item:tau-fixes-idempotents}.

Finally, let $u\in\PosId(\mathbf M)$ with $u\le \tau(x)$.
By isotonicity, $xu\le x\tau(x)=x$ and $ux\le \tau(x)x=x$.
Since $u$ is positive, also $x\le xu$ and $x\le ux$.
Therefore $xu=x=ux$.
This proves \eqref{item:below-tau-gives-unit}.
\end{proof}

\begin{lemma}\label{lem:idempotents-max}
Let $p,q\in\PosId(\mathbf M)$.
Then \(pq=qp=\max\{p,q\}.\)
In particular, the multiplication on $\PosId(\mathbf M)$ coincides with the maximum operation.
\end{lemma}

\begin{proof}
Because the order is total, either $p\le q$ or $q\le p$.
Assume first that $p\le q$.
Since $p$ is positive, $q\le pq.$
By isotonicity and $p\le q$, $pq\le qq=q.$
Hence $pq=q=\max\{p,q\}$.
Similarly, $q=qe\leq qp\leq qq=q$, so $qp=q$.
The case $q\leq p$ is symmetric.
\end{proof}

\begin{proposition}\label{prop:tau-of-product}
Let $\mathbf M$ be local-unit-aligned.
Then for all $x,y\in M$, \(\tau(xy)\ge \tau(x)\tau(y)=\max\{\tau(x),\tau(y)\}.\)
\end{proposition}

\begin{proof}
Since $y\tau(y)=y$, we have $xy\tau(y)=x(y\tau(y))=xy$, so $\tau(y)\in R(xy)$ and therefore $\tau(y)\le \tau(xy)$.
Likewise, since $\tau(x)x=x$, $\tau(x)xy=(\tau(x)x)y=xy$, so $\tau(x)\in L(xy)$ and therefore $\tau(x)\le \tau(xy)$.
By Lemma~\ref{lem:idempotents-max}, $\tau(x)\tau(y)=\max\{\tau(x),\tau(y)\}\le \tau(xy).$
\end{proof}

\begin{definition}
Let $\mathbf M$ be local-unit-aligned.

\begin{enumerate}
\item For $u\in\PosId(\mathbf M)$, the \emph{$u$-layer} is \(\layer{u}:=\{x\in M:\tau(x)=u\}.\)
\item For $A\subseteq\PosId(\mathbf M)$, put \(\layer{A}:=\tau^{-1}(A)=\{x\in M:\tau(x)\in A\}.\)
\item For $A,B\subseteq\PosId(\mathbf M)$, define their \emph{$\tau$-saturated product} by
\[
 \tprod{A}{B}:=\tau(\layer{A}\cdot\layer{B})
 =\{\tau(xy):x\in\layer{A},\ y\in\layer{B}\}.
\]
\end{enumerate}
\end{definition}

\begin{lemma}\label{lem:tprod-monotone}
If $A\subseteq A'$ and $B\subseteq B'$, then \(\tprod{A}{B}\subseteq \tprod{A'}{B'}.\)
In particular, \(A\subseteq A'\Longrightarrow \tprod{A}{A}\subseteq \tprod{A'}{A'}.\)
\end{lemma}

\begin{proof}
From $A\subseteq A'$ and $B\subseteq B'$ we get $\layer{A}\subseteq \layer{A'}$ and $\layer{B}\subseteq \layer{B'}$.
Hence $\layer{A}\cdot\layer{B}\subseteq \layer{A'}\cdot\layer{B'}.$
Applying $\tau$ yields the claim.
\end{proof}

\begin{remark}\label{rem:A-subset-tprodAA}
For every $A\subseteq\PosId(\mathbf M)$, \(A\subseteq \tprod{A}{A}.\)
Indeed, if $u\in A$, then $u\in \layer{A}$ by Lemma~\ref{lem:basic_tau}\eqref{item:tau-fixes-idempotents}, and \(u=\tau(u^{2})\in \tprod{A}{A}.\)
\end{remark}

\begin{remark}[Guide to the terminology]
The terminology used below operates at three related levels: subsets and partitions of $\PosId(\mathbf M)$, the corresponding decompositions of the ambient monoid $\mathbf M$, and properties of $\mathbf M$ itself.
Since the corresponding terms are closely related, we summarize them here for convenience (Table~\ref{terminology}).
\begin{table}[ht]
\centering
\begin{tabular}{|p{0.24\textwidth}|p{0.25\textwidth}|p{0.33\textwidth}|}
\hline
\textbf{Object} & \textbf{Terminology} & \textbf{Characterization} \\
\hline
decomposition $\{\layer{A}:A\in\mathcal P\}$ of $M$
& decomposition induced by the partition $\mathcal P$ of $\PosId(\mathbf M)$
& the family of layers attached to the blocks of $\mathcal P$ \\
\hline
subset $A\subseteq \PosId(\mathbf M)$
& $\tau$-stable
& $\tprod{A}{A}=A$ \\
\hline
partition $\mathcal P$ of $\PosId(\mathbf M)$
& $\tau$-stable
& every block of $\mathcal P$ is $\tau$-stable \\
\hline
canonical partition $\Ctau(\mathbf M)$
& $\tau$-cohesive partition
& the finest $\tau$-stable partition of $\PosId(\mathbf M)$ \\
\hline
canonical decomposition $\Ltau(\mathbf M)$
& $\tau$-cohesive decomposition
& the decomposition of $\mathbf M$ induced by $\Ctau(\mathbf M)$ \\
\hline
monoid $\mathbf M$
& $\tau$-cohesive
& $\Ctau(\mathbf M)$ is trivial ($\Ctau(\mathbf M)=\{\PosId(\mathbf M)\}$), equivalently $\Ltau(\mathbf M)$ is trivial ($\Ltau(\mathbf M)=\{M\}$) \\
\hline
partition $\mathcal P$ of $\PosId(\mathbf M)$
& $\tau$-multiplication-coherent
& for all blocks $A,B\in\mathcal P$, the set $\tprod{A}{B}$ is contained in a unique block of $\mathcal P$ \\
\hline
canonical partition $\Cm(\mathbf M)$
& $\tau$-multiplication-coherent partition
& the finest $\tau$-multiplication-coherent partition of $\PosId(\mathbf M)$ \\
\hline
canonical decomposition $\Lm(\mathbf M)$
& $\tau$-multiplication-coherent decomposition
& the decomposition of $\mathbf M$ induced by $\Cm(\mathbf M)$ \\
\hline
monoid $\mathbf M$
& $\tau$-multiplication-cohesive
& $\Cm(\mathbf M)$ is trivial ($\Cm(\mathbf M)=\{\PosId(\mathbf M)\}$), equivalently $\Lm(\mathbf M)$ is trivial ($\Lm(\mathbf M)=\{M\}$) \\
\hline
\end{tabular}
\caption{Terminology summary}
\label{terminology} 
\end{table}
Thus $\tau$-stable and $\tau$-multiplication-coherent are properties of subsets and partitions on $\PosId(\mathbf M)$; the corresponding induced decompositions live on the ambient monoid $M$; and $\tau$-cohesive and $\tau$-multiplication-cohesive are properties of the monoid itself, defined through the triviality of the corresponding canonical decompositions (equivalently, of the corresponding canonical partitions).
\end{remark}

\section{\texorpdfstring{$\tau$-stable partitions and the $\tau$-cohesive decomposition}{tau-stable partitions and the tau-cohesive decomposition}}

We use the standard refinement order on partitions: a partition $\mathcal P$ of a set $X$ \emph{refines} a partition $\mathcal Q$ of $X$ if every block of $\mathcal P$ is contained in a block of $\mathcal Q$.

\begin{definition}
Let $\mathbf M$ be local-unit-aligned.

\begin{enumerate}
\item A nonempty set $A\subseteq\PosId(\mathbf M)$ is called \emph{$\tau$-stable} if \(\tprod{A}{A}=A.\)
\item A partition $\mathcal P$ of $\PosId(\mathbf M)$ is called \emph{$\tau$-stable} if every block of $\mathcal P$ is $\tau$-stable.
\end{enumerate}
\end{definition}

Thus a \(\tau\)-stable block is meant to isolate a collection of local units whose associated layers are closed under taking the local unit of an internal product.

\begin{definition}
Let $\mathcal P$ be a partition of $\PosId(\mathbf M)$.
Define a graph on the set of blocks of $\mathcal P$ by joining two blocks $A$ and $B$ whenever \(\tprod{A}{A}\cap \tprod{B}{B}\neq\varnothing.\)
We denote by $\CCtau(\mathcal P)$ the partition obtained by merging precisely those blocks that lie in the same connected component of this graph.
\end{definition}

\begin{lemma}\label{lem:CCtau-monotone}
If a partition $\mathcal P$ refines a partition $\mathcal Q$, then $\CCtau(\mathcal P)$ refines $\CCtau(\mathcal Q)$.
\end{lemma}

\begin{proof}
Let $A,B$ be blocks of $\mathcal P$ with $\tprod{A}{A}\cap \tprod{B}{B}\neq\varnothing.$
Let $A',B'$ be the blocks of $\mathcal Q$ containing $A$ and $B$, respectively.
By Lemma~\ref{lem:tprod-monotone}, $\tprod{A}{A}\subseteq \tprod{A'}{A'}$ and $\tprod{B}{B}\subseteq \tprod{B'}{B'}$.
Hence $\tprod{A'}{A'}\cap \tprod{B'}{B'}\neq\varnothing.$
If \(A'=B'\), then the two image blocks already coincide.
If \(A'\neq B'\), the displayed nonempty intersection shows that \(A'\) and \(B'\) are joined by an edge in the graph of \(\mathcal Q\).
Thus every edge in the graph of \(\mathcal P\) maps either to a single vertex or to an edge in the graph of \(\mathcal Q\).
Therefore connected components for \(\mathcal P\) are contained in connected components for \(\mathcal Q\), which exactly says that \(\CCtau(\mathcal P)\) refines \(\CCtau(\mathcal Q)\).
\end{proof}

\begin{proposition}\label{prop:CCtau-fixed-points}
For a partition $\mathcal P$ of $\PosId(\mathbf M)$, the following are equivalent:
\begin{enumerate}
\item\label{item:CCtau-fixed} $\CCtau(\mathcal P)=\mathcal P$;
\item\label{item:CCtau-stable} $\mathcal P$ is $\tau$-stable.
\end{enumerate}
\end{proposition}

\begin{proof}
Assume first that $\CCtau(\mathcal P)=\mathcal P$ and let $A\in\mathcal P$.
Take $u\in \tprod{A}{A}$ and let $B\in\mathcal P$ be the unique block containing $u$.
By Remark~\ref{rem:A-subset-tprodAA}, $B\subseteq \tprod{B}{B}.$
Since $u\in \tprod{A}{A}\cap B$, it follows that $\tprod{A}{A}\cap \tprod{B}{B}\neq\varnothing.$
Hence $A$ and $B$ are joined by an edge, so they would be merged by $\CCtau$.
Because $\CCtau(\mathcal P)=\mathcal P$, we must have $A=B$.
Therefore every element of $\tprod{A}{A}$ lies in $A$, that is, $\tprod{A}{A}\subseteq A.$
Together with Remark~\ref{rem:A-subset-tprodAA}, this gives $\tprod{A}{A}=A$.
Thus $\mathcal P$ is $\tau$-stable.

Conversely, assume that $\mathcal P$ is $\tau$-stable.
Then for distinct blocks $A,B\in\mathcal P$, $\tprod{A}{A}=A$, $\tprod{B}{B}=B$, so $\tprod{A}{A}\cap \tprod{B}{B}=A\cap B=\varnothing.$
Hence the graph defining $\CCtau(\mathcal P)$ has no edges between distinct blocks, so $\CCtau(\mathcal P)=\mathcal P$.
\end{proof}

\begin{definition}
Let $\mathbf M$ be local-unit-aligned.
Start from the singleton partition \(\mathcal P_{0}:=\bigl\{\{u\}:u\in\PosId(\mathbf M)\bigr\},\) and define recursively \(\mathcal P_{n+1}:=\CCtau(\mathcal P_{n})\) for \(n\ge 0\).
Since $\PosId(\mathbf M)$ is finite and every step only merges blocks, the sequence stabilizes: there exists $N$ such that $\mathcal P_{N+1}=\mathcal P_{N}$.
We write \(\Ctau(\mathbf M):=\mathcal P_{N}\) for the stable partition and call it the \emph{$\tau$-cohesive partition} of $\mathbf M$.
The induced family \(\Ltau(\mathbf M):=\{\layer{A}:A\in\Ctau(\mathbf M)\}\) is the \emph{$\tau$-cohesive decomposition} of $\mathbf M$.
\end{definition}

\begin{theorem}\label{thm:finest-tau-stable}
The partition $\Ctau(\mathbf M)$ is the finest $\tau$-stable partition of $\PosId(\mathbf M)$.
Equivalently, $\Ltau(\mathbf M)$ is the finest decomposition of $M$ induced by a $\tau$-stable partition.
\end{theorem}

\begin{proof}
By construction, $\Ctau(\mathbf M)$ is a fixed point of $\CCtau$, hence $\tau$-stable by Proposition~\ref{prop:CCtau-fixed-points}, \eqref{item:CCtau-fixed}\(\Rightarrow\)\eqref{item:CCtau-stable}.

Let $\mathcal Q$ be any $\tau$-stable partition of $\PosId(\mathbf M)$.
Then $\CCtau(\mathcal Q)=\mathcal Q$ by Proposition~\ref{prop:CCtau-fixed-points}, \eqref{item:CCtau-stable}\(\Rightarrow\)\eqref{item:CCtau-fixed}.
The singleton partition $\mathcal P_0$ refines $\mathcal Q$, so Lemma~\ref{lem:CCtau-monotone} yields inductively that every $\mathcal P_n$ refines $\mathcal Q$.
Therefore the stable partition $\Ctau(\mathbf M)$ also refines $\mathcal Q$.
Thus $\Ctau(\mathbf M)$ refines every $\tau$-stable partition, so it is the finest one.
The equivalent statement about the induced decomposition follows by taking $\tau^{-1}$ of the blocks.
\end{proof}

\begin{definition}\label{def:tau-cohesive}
A finite local-unit-aligned totally ordered monoid $\mathbf M$ is called \emph{$\tau$-cohesive} if its $\tau$-cohesive partition is trivial, that is, \(\Ctau(\mathbf M)=\{\PosId(\mathbf M)\}.\)
Equivalently, its $\tau$-cohesive decomposition is trivial, that is, \(\Ltau(\mathbf M)=\{M\}.\)
\end{definition}

\begin{remark}\label{rem:tau-cohesive-naming-forward}
The terminology introduced here is intentionally two-level.
For every finite local-unit-aligned totally ordered monoid $\mathbf M$, the notation \(\Ctau(\mathbf M)\) denotes a canonical partition of $\PosId(\mathbf M)$, called the \emph{$\tau$-cohesive partition}, whether or not it is trivial.
By contrast, the monoid $\mathbf M$ itself is called \emph{$\tau$-cohesive} only when this canonical partition has a single block.
The reason for this naming convention becomes conceptually clear later: Theorem~\ref{thm:Ctau-characterized-by-components} shows that \(\Ctau(\mathbf M)\) is exactly the unique $\tau$-stable partition of \(\PosId(\mathbf M)\) such that every associated component monoid is already intrinsically $\tau$-cohesive.
\end{remark}

\section{\texorpdfstring{$\tau$-multiplication-coherent partitions and the $\tau$-multiplication-coherent decomposition}{tau-multiplication-coherent partitions and the tau-multiplication-coherent decomposition}}

\begin{definition}
A partition $\mathcal P$ of $\PosId(\mathbf M)$ is called \emph{$\tau$-multiplication-coherent} if for all blocks $A,B\in\mathcal P$ there exists a unique block $C\in\mathcal P$ such that \(\tprod{A}{B}\subseteq C.\)
Equivalently, since $\tprod{A}{B}\neq\varnothing$ for all blocks $A,B\in\mathcal P$, the partition $\mathcal P$ is $\tau$-multiplication-coherent if for all blocks $A,B,C,D\in\mathcal P$,
\[
 C\cap \tprod{A}{B}\neq\varnothing
 \quad\text{and}\quad
 D\cap \tprod{A}{B}\neq\varnothing
 \Longrightarrow C=D.
\]
\end{definition}

This condition isolates precisely when the \(\tau\)-component of a product is determined by the two input blocks alone, rather than by the particular representatives chosen inside their layers.

\begin{lemma}\label{lem:coherent-implies-stable}
Every $\tau$-multiplication-coherent partition is $\tau$-stable.
\end{lemma}

\begin{proof}
Let $\mathcal P$ be $\tau$-multiplication-coherent and let $A\in\mathcal P$.
By Remark~\ref{rem:A-subset-tprodAA}, the set $\tprod{A}{A}$ meets $A$.
By coherence, $\tprod{A}{A}$ is contained in a unique block of $\mathcal P$; since it meets $A$, that unique block must be $A$ itself.
Thus $\tprod{A}{A}\subseteq A.$
Together with Remark~\ref{rem:A-subset-tprodAA}, this yields $\tprod{A}{A}=A$.
Hence $\mathcal P$ is $\tau$-stable.
\end{proof}

\begin{definition}
Let $\mathcal P$ be a partition of $\PosId(\mathbf M)$.
Define a graph on the set of blocks of $\mathcal P$ by joining two blocks $C$ and $D$ whenever there exist blocks $A,B\in\mathcal P$ such that
\[
 C\cap \tprod{A}{B}\neq\varnothing
 \qquad\text{and}\qquad
 D\cap \tprod{A}{B}\neq\varnothing.
\]
We denote by $\CCm(\mathcal P)$ the partition obtained by merging precisely those blocks that lie in the same connected component of this graph.
\end{definition}

\begin{lemma}\label{lem:CCx-monotone}
If a partition $\mathcal P$ refines a partition $\mathcal Q$, then $\CCm(\mathcal P)$ refines $\CCm(\mathcal Q)$.
\end{lemma}

\begin{proof}
Let $A,B,C,D$ be blocks of $\mathcal P$ such that $C\cap \tprod{A}{B}\neq\varnothing$ and $D\cap \tprod{A}{B}\neq\varnothing$.
Let $A',B',C',D'$ be the blocks of $\mathcal Q$ containing $A,B,C,D$, respectively.
By Lemma~\ref{lem:tprod-monotone}, $\tprod{A}{B}\subseteq \tprod{A'}{B'}.$
Hence $C'\cap \tprod{A'}{B'}\neq\varnothing$ and $D'\cap \tprod{A'}{B'}\neq\varnothing$.
If \(C'=D'\), then the two image blocks already coincide.
If \(C'\neq D'\), the displayed intersections show that \(C'\) and \(D'\) are joined by an edge in the graph of \(\mathcal Q\).
Thus every edge in the graph of \(\mathcal P\) maps either to a single vertex or to an edge in the graph of \(\mathcal Q\).
Therefore connected components for \(\mathcal P\) are contained in connected components for \(\mathcal Q\), so \(\CCm(\mathcal P)\) refines \(\CCm(\mathcal Q)\).
\end{proof}

\begin{proposition}\label{prop:CCx-fixed-points}
For a partition $\mathcal P$ of $\PosId(\mathbf M)$, the following are equivalent:
\begin{enumerate}
\item\label{item:CCx-fixed} $\CCm(\mathcal P)=\mathcal P$;
\item\label{item:CCx-coherent} $\mathcal P$ is $\tau$-multiplication-coherent.
\end{enumerate}
\end{proposition}

\begin{proof}
Assume first that $\CCm(\mathcal P)=\mathcal P$.
Let $A,B\in\mathcal P$, and let $C,D\in\mathcal P$ satisfy $C\cap \tprod{A}{B}\neq\varnothing$ and $D\cap \tprod{A}{B}\neq\varnothing$.
Then $C$ and $D$ are joined by an edge in the graph defining $\CCm(\mathcal P)$, hence they would be merged.
Since the partition is unchanged, we must have $C=D$.
Thus, since \(\tprod{A}{B}\neq\varnothing\), the set \(\tprod{A}{B}\) is contained in a unique block of \(\mathcal P\), proving coherence.

Conversely, assume that $\mathcal P$ is $\tau$-multiplication-coherent.
If $C$ and $D$ are distinct blocks of $\mathcal P$, then there are no blocks $A,B\in\mathcal P$ such that both $C$ and $D$ meet $\tprod{A}{B}$, because coherence would then force $C=D$.
Hence the graph defining $\CCm(\mathcal P)$ has no edges between distinct blocks, and therefore $\CCm(\mathcal P)=\mathcal P$.
\end{proof}

\begin{definition}
Let $\mathbf M$ be local-unit-aligned.
Start again from the singleton partition \(\mathcal Q_{0}:=\bigl\{\{u\}:u\in\PosId(\mathbf M)\bigr\},\) and define recursively \(\mathcal Q_{n+1}:=\CCm(\mathcal Q_{n})\) for \(n\ge 0\).
Since $\PosId(\mathbf M)$ is finite and every step only merges blocks, there exists $N$ such that $\mathcal Q_{N+1}=\mathcal Q_{N}$.
We write \(\Cm(\mathbf M):=\mathcal Q_{N}\) for the fixed point and call it the \emph{$\tau$-multiplication-coherent partition} of $\mathbf M$.
The induced family \(\Lm(\mathbf M):=\{\layer{A}:A\in\Cm(\mathbf M)\}\) is the \emph{$\tau$-multiplication-coherent decomposition} of $\mathbf M$.
\end{definition}

\begin{definition}\label{def:tau-multiplication-cohesive}
A finite local-unit-aligned totally ordered monoid $\mathbf M$ is called \emph{$\tau$-multiplication-cohesive} if its $\tau$-multiplication-coherent partition is trivial, that is, \(\Cm(\mathbf M)=\{\PosId(\mathbf M)\}.\)
Equivalently, its $\tau$-multiplication-coherent decomposition is trivial, that is, it has a single component: \(\Lm(\mathbf M)=\{M\}.\)
\end{definition}

\begin{theorem}\label{thm:finest-coherent}
The partition $\Cm(\mathbf M)$ is the finest $\tau$-multiplication-coherent partition of $\PosId(\mathbf M)$.
\end{theorem}

\begin{proof}
By construction, $\Cm(\mathbf M)$ is a fixed point of $\CCm$, hence it is $\tau$-multiplication-coherent by Proposition~\ref{prop:CCx-fixed-points}, \eqref{item:CCx-fixed}\(\Rightarrow\)\eqref{item:CCx-coherent}.

Let $\mathcal R$ be any $\tau$-multiplication-coherent partition of $\PosId(\mathbf M)$.
Then $\CCm(\mathcal R)=\mathcal R$ by Proposition~\ref{prop:CCx-fixed-points}, \eqref{item:CCx-coherent}\(\Rightarrow\)\eqref{item:CCx-fixed}.
The singleton partition $\mathcal Q_0$ refines $\mathcal R$, so Lemma~\ref{lem:CCx-monotone} yields inductively that every $\mathcal Q_n$ refines $\mathcal R$.
Therefore the stable partition $\Cm(\mathbf M)$ also refines $\mathcal R$.
Thus $\Cm(\mathbf M)$ refines every $\tau$-multiplication-coherent partition, so it is the finest one.
\end{proof}

\begin{theorem}\label{thm:tau-refines-mult}
The $\tau$-cohesive partition refines the $\tau$-multiplication-coherent partition: \(\Ctau(\mathbf M)\ \text{refines}\ \Cm(\mathbf M).\)
Consequently, every block of $\Cm(\mathbf M)$ is a union of blocks of $\Ctau(\mathbf M)$, and every component of $\Lm(\mathbf M)$ is a union of components of $\Ltau(\mathbf M)$.
\end{theorem}

\begin{proof}
By Theorem~\ref{thm:finest-coherent}, the partition $\Cm(\mathbf M)$ is $\tau$-multiplication-coherent.
By Lemma~\ref{lem:coherent-implies-stable}, it is therefore $\tau$-stable.
Now Theorem~\ref{thm:finest-tau-stable} says that $\Ctau(\mathbf M)$ refines every $\tau$-stable partition.
In particular, it refines $\Cm(\mathbf M)$.
The statements about unions follow by taking unions of the corresponding blocks and their $\tau$-preimages.
\end{proof}

\begin{corollary}\label{cor:tau-cohesive-implies-multiplication-cohesive}
Every $\tau$-cohesive finite local-unit-aligned totally ordered monoid is $\tau$-multiplication-cohesive.
\end{corollary}

\begin{proof}
Let $\mathbf M$ be $\tau$-cohesive.
Then $\Ctau(\mathbf M)=\{\PosId(\mathbf M)\}.$
By Theorem~\ref{thm:tau-refines-mult}, the partition $\Ctau(\mathbf M)$ refines $\Cm(\mathbf M)$.
Since $\Ctau(\mathbf M)$ has only one block, this is possible only if $\Cm(\mathbf M)=\{\PosId(\mathbf M)\}.$
Hence $\mathbf M$ is $\tau$-multiplication-cohesive.
\end{proof}

\begin{remark}
The opposite direction of Corollary~\ref{cor:tau-cohesive-implies-multiplication-cohesive} is not true.
The smallest counterexample is shown in Fig.~\ref{fig:smallest-tau-mult-but-not-tau}.
The minimality assertion is computational: exhaustive enumeration of all smaller finite local-unit-aligned totally ordered monoids shows that none is $\tau$-multiplication-cohesive without being $\tau$-cohesive.
Its set of positive idempotents is $\PosId(\mathbf M)=\{1,2,4\}$.
The $\tau$-cohesive partition is $\Ctau(\mathbf M)=\{\{1,4\},\{2\}\}$, yielding the decomposition $\Ltau(\mathbf M)=\{\{1,3,4\},\{2\}\}$, so $\mathbf M$ is not $\tau$-cohesive.
In contrast, the $\tau$-multiplication-coherent partition is trivial, $\Cm(\mathbf M)=\{\{1,2,4\}\}$, and hence $\Lm(\mathbf M)=\{\{1,2,3,4\}\}$, so $\mathbf M$ is $\tau$-multiplication-cohesive.
\begin{figure}[ht]
\centering
\[
\mathbf M=
\left(
\begin{array}{cccc}
1 & 2 & 3 & 4\\
2 & 2 & 4 & 4\\
3 & 4 & 4 & 4\\
4 & 4 & 4 & 4
\end{array}
\right)
\qquad
\tau=
\left(
\begin{array}{c}
1\\
2\\
1\\
4
\end{array}
\right)
\]
\caption{Multiplication table of a $\tau$-multiplication-cohesive but not $\tau$-cohesive monoid $\mathbf M$, together with its local-unit map $\tau$.}
\label{fig:smallest-tau-mult-but-not-tau}
\end{figure}
\end{remark}

\section{Component monoids}

\begin{proposition}\label{prop:tau-stable-gives-subsemigroup}
Let $\mathcal P$ be a $\tau$-stable partition of $\PosId(\mathbf M)$ and let $A\in\mathcal P$.
Then $\layer{A}$ is closed under multiplication.
If \(e_A:=\min A,\) then $e_A\in\layer{A}$ and \(xe_A=x=e_Ax\) for \(x\in\layer{A}\).
Hence $\layer{A}$ is a subsemigroup of $\mathbf M$ with identity element $e_A$.
In particular, the identity element of the component monoid carried by $A$ is \(e_A=\min A.\)
\end{proposition}

\begin{proof}
Let $x,y\in\layer{A}$.
Then $\tau(x),\tau(y)\in A$, so $\tau(xy)\in \tprod{A}{A}=A$ because $A$ is $\tau$-stable.
Hence $xy\in\layer{A}$, so $\layer{A}$ is closed under multiplication.

Since $A\subseteq\PosId(\mathbf M)$ and the chain is finite, $e_A:=\min A$ exists and belongs to $A$.
By Lemma~\ref{lem:basic_tau}\eqref{item:tau-fixes-idempotents}, $\tau(e_A)=e_A$, so $e_A\in\layer{A}$.
Now if $x\in\layer{A}$, then $\tau(x)\in A$, hence $e_A\le \tau(x)$ by minimality of $e_A$.
Lemma~\ref{lem:basic_tau}\eqref{item:below-tau-gives-unit} therefore gives $xe_A=x=e_Ax.$
Thus $e_A$ is the identity element of $\layer{A}$.
\end{proof}

\begin{remark}[Canonical blocks and block units]\label{rem:canonical-blocks-block-units}
For the distinguished partition \(\Cm(\mathbf M)\), we call its blocks the \emph{canonical blocks} of \(\mathbf M\).
If \(A\in\Cm(\mathbf M)\), then the identity element \(e_A=\min A\) of the associated component monoid \(\mathbf M_A\) will be called the \emph{block unit} of \(A\).
\end{remark}

\begin{remark}[Intrinsic conventions for component monoids]\label{rem:intrinsic-conventions}
Unless explicitly stated otherwise, all constructions attached to a finite local-unit-aligned totally ordered monoid \(\mathbf N\) are computed intrinsically in \(\mathbf N\).
Thus positivity is measured with respect to the unit of \(\mathbf N\), the local-unit map is \(\tau_{\mathbf N}\), the set \(\PosId(\mathbf N)\) and the layers are formed inside \(\mathbf N\), and the partitions \(\Ctau(\mathbf N)\) and \(\Cm(\mathbf N)\), together with the decompositions \(\Ltau(\mathbf N)\) and \(\Lm(\mathbf N)\), are intrinsic to \(\mathbf N\).
In particular, after passing to a component monoid \(\mathbf M_A\), every subsequent application of these constructions is taken inside \(\mathbf M_A\), whose identity element is \(e_A\).
\end{remark}

\begin{remark}\label{rem:ambient-vs-intrinsic-multiplication-cohesive}
A priori, a component of an ambient $\tau$-multiplication-coherent decomposition need not be $\tau$-multiplication-cohesive when regarded as a monoid in its own right.
The reason is that $\tau$-multiplication-coherence is initially an ambient notion.
In the ambient monoid, a block may be forced to remain undivided by mixed products involving elements outside that block.
Once the corresponding component monoid is isolated, those external products are no longer present, so the intrinsic $\tau$-multiplication-coherent decomposition of the component may become strictly finer than the restriction of the ambient one.

Thus one must distinguish carefully between ambient indecomposability and intrinsic $\tau$-multiplication-cohesiveness.
In the present finite setting, however, the canonical components will later be shown to be intrinsically $\tau$-multiplication-cohesive.
\end{remark}

\begin{theorem}\label{thm:component-local-unit-aligned-monoid}
Let $\mathcal P$ be a $\tau$-stable partition of $\PosId(\mathbf M)$ and let $A\in\mathcal P$.
Then
\[
 \mathbf M_{A}:=\langle \layer{A},\le\!\upharpoonright_{\layer{A}},\cdot\!\upharpoonright_{\layer{A}\times\layer{A}},e_A\rangle
\]
is a finite local-unit-aligned totally ordered monoid.
Its local-unit map is the restriction of the ambient one: for every $x\in\layer{A}$, \(\tau_{\mathbf M_A}(x)=\tau_{\mathbf M}(x)=\tau(x).\)
\end{theorem}

\begin{proof}
By Proposition~\ref{prop:tau-stable-gives-subsemigroup}, $\mathbf M_A$ is a finite totally ordered monoid.
It remains to identify its local-unit map and to prove it is local-unit-aligned.

Let $x\in\layer{A}$.
Since $\tau(x)\in A\subseteq\PosId(\mathbf M)$, Lemma~\ref{lem:basic_tau}\eqref{item:tau-fixes-idempotents} gives $\tau(\tau(x))=\tau(x)$, so $\tau(x)\in\layer{A}$.
Moreover, $x\tau(x)=x=\tau(x)x$ by Lemma~\ref{lem:basic_tau}\eqref{item:local-unit-equalities}.

Hence $\tau(x)$ belongs to both internal cuts $\{z\in\layer{A}:xz\le x\}$ and $\{z\in\layer{A}:zx\le x\}$.

Now let $z\in\layer{A}$ satisfy $xz\le x$.
Since $z\in M$, the definition of the ambient local unit gives $z\le \tau(x)$.
Thus $\tau(x)$ is the greatest element of the internal right cut.
The same argument, using elements $z\in\layer{A}$ satisfying $zx\le x$, shows that $\tau(x)$ is also the greatest element of the internal left cut.
Therefore the two internal maxima coincide and equal $\tau(x)$.

Finally, since $\tau(x)\in A$ and $e_A=\min A$, we have $e_A\le \tau(x).$
Thus $\tau(x)$ is positive in $\mathbf M_A$.
This proves that $\mathbf M_A$ is local-unit-aligned and that its local-unit map is the restriction of the ambient one.
\end{proof}

\begin{corollary}\label{cor:all-components-local-unit-aligned-monoids}
Every component of the $\tau$-cohesive decomposition and every component of the $\tau$-multiplication-coherent decomposition is a finite local-unit-aligned totally ordered monoid.
\end{corollary}

\begin{proof}
The $\tau$-cohesive partition is $\tau$-stable by Theorem~\ref{thm:finest-tau-stable}.
The $\tau$-multiplication-coherent partition is $\tau$-multiplication-coherent by Theorem~\ref{thm:finest-coherent}, hence $\tau$-stable by Lemma~\ref{lem:coherent-implies-stable}.
Now apply Theorem~\ref{thm:component-local-unit-aligned-monoid}.
\end{proof}

\begin{lemma}\label{lem:intrinsic-tau-on-component}
Let $\mathcal P$ be a $\tau$-stable partition of $\PosId(\mathbf M)$, let $A\in\mathcal P$, and let
\[
\mathbf M_A
=
\langle \layer{A},
\le\!\upharpoonright_{\layer{A}},
\cdot\!\upharpoonright_{\layer{A}\times\layer{A}},
e_A\rangle
\]
be the corresponding component monoid.

Then the following hold.

\begin{enumerate}
\item\label{item:component-posid} \(\PosId(\mathbf M_A)=A.\)
\item\label{item:component-layers} For every $U\subseteq A$, the $U$-layer computed in $\mathbf M_A$ coincides with the ambient one:
\[
\{x\in \layer{A}:\tau_{\mathbf M_A}(x)\in U\}
=
\{x\in M:\tau_{\mathbf M}(x)\in U\}.
\]

\item\label{item:component-tprod} For all $U,V\subseteq A$, the $\tau$-saturated product computed in $\mathbf M_A$ coincides with the ambient one: \((\tprod{U}{V})^{\mathbf M_A} = (\tprod{U}{V})^{\mathbf M}.\)
\end{enumerate}
\end{lemma}

\begin{proof}
By Theorem~\ref{thm:component-local-unit-aligned-monoid}, for every $x\in \layer{A}$, \(\tau_{\mathbf M_A}(x)=\tau_{\mathbf M}(x)=\tau(x).\)

To prove \eqref{item:component-posid}, first let $u\in A$.
Then $u\in\PosId(\mathbf M)$ and $u\in \layer{A}$.
Also $e_A=\min A\le u$, so $u$ is positive in $\mathbf M_A$, and of course it is still idempotent.
Hence $u\in\PosId(\mathbf M_A)$.

Conversely, let \(u\in\PosId(\mathbf M_A)\).
By Lemma~\ref{lem:basic_tau}\eqref{item:tau-fixes-idempotents}, applied inside \(\mathbf M_A\), we have \(\tau_{\mathbf M_A}(u)=u\).
Since \(u\in\layer{A}\), the displayed identity gives \(u=\tau_{\mathbf M}(u)\in A.\)
Thus $\PosId(\mathbf M_A)=A$.

For \eqref{item:component-layers}, let \(U\subseteq A\).
If \(x\in\layer{A}\), then $ \tau_{\mathbf M_A}(x)\in U \iff \tau_{\mathbf M}(x)\in U. $
This proves the equality after observing that the right-hand side is already contained in \(\layer{A}\): indeed, if \(\tau_{\mathbf M}(x)\in U\subseteq A\), then \(x\in\layer{A}\).

Finally, let \(U,V\subseteq A\).
By \eqref{item:component-layers}, the \(U\)-layer and \(V\)-layer computed in \(\mathbf M_A\) are exactly the ambient \(U\)-layer and \(V\)-layer.
In particular, these layers are contained in \(\layer{A}=M_A\).
Hence, if \(x\) lies in the \(U\)-layer and \(y\) lies in the \(V\)-layer, then \(x,y\in M_A\), and since \(M_A\) is a submonoid of \(\mathbf M\), the product \(xy\) computed in \(\mathbf M_A\) is the same as the ambient product and still belongs to \(M_A\).
Moreover \(\tau_{\mathbf M_A}(xy)=\tau_{\mathbf M}(xy)\) by Theorem~\ref{thm:component-local-unit-aligned-monoid}.
Therefore the two sets of possible local units \(\tau(xy)\), and hence the two \(\tau\)-saturated products, coincide: \((\tprod{U}{V})^{\mathbf M_A} = (\tprod{U}{V})^{\mathbf M}.\)
\end{proof}

\begin{lemma}[Component calculus]\label{lem:component-calculus}
Let $\mathcal P$ be a $\tau$-stable partition of $\PosId(\mathbf M)$, let $A\in\mathcal P$, and let \(\mathbf M_A=\langle\layer{A},\le\!\upharpoonright_{\layer{A}}, \cdot\!\upharpoonright_{\layer{A}\times\layer{A}},e_A\rangle\) be the corresponding component monoid.
Then the following hold.
\begin{enumerate}
\item\label{item:component-calculus-submonoid} $\mathbf M_A$ is an ordered submonoid of $\mathbf M$ with identity $e_A=\min A$.
\item\label{item:component-calculus-tau} For every $x\in M_A$, the intrinsic and ambient local-unit maps agree: \(\tau_{\mathbf M_A}(x)=\tau_{\mathbf M}(x)\in A\).
\item\label{item:component-calculus-posid} $\PosId(\mathbf M_A)=A$.
\item\label{item:component-calculus-layers} For every $U\subseteq A$, the $U$-layer computed in $\mathbf M_A$ is the same as the ambient $U$-layer.
\item\label{item:component-calculus-tprod} For all $U,V\subseteq A$, the $\tau$-saturated product computed in $\mathbf M_A$ is the same as the ambient one: \((\tprod{U}{V})^{\mathbf M_A}=(\tprod{U}{V})^{\mathbf M}\).
\end{enumerate}
Consequently, whenever only subsets of $A$ are involved, local-unit, layer, order, multiplication, and $\tau$-saturated-product computations may be read either inside the component monoid $\mathbf M_A$ or in the ambient monoid $\mathbf M$.
\end{lemma}

\begin{proof}
Item~\ref{item:component-calculus-submonoid} is Proposition~\ref{prop:tau-stable-gives-subsemigroup}; item~\ref{item:component-calculus-tau} is Theorem~\ref{thm:component-local-unit-aligned-monoid}; and items~\ref{item:component-calculus-posid}--\ref{item:component-calculus-tprod} are, respectively, Lemma~\ref{lem:intrinsic-tau-on-component}\eqref{item:component-posid}--\eqref{item:component-tprod}.
\end{proof}

\begin{theorem}\label{thm:Ctau-characterized-by-components}
Let $\mathcal P$ be a $\tau$-stable partition of $\PosId(\mathbf M)$.
For each $A\in\mathcal P$, let $\mathbf M_A$ be the corresponding component monoid.

Then the following are equivalent.

\begin{enumerate}
\item\label{item:Ctau-char-1} $\mathcal P$ is the $\tau$-cohesive partition of $\mathbf M$, that is, \( \mathcal P=\Ctau(\mathbf M). \)

\item\label{item:Ctau-char-2} For every block $A\in\mathcal P$, the component monoid $\mathbf M_A$ is $\tau$-cohesive.
\end{enumerate}
\end{theorem}

\begin{proof}
Assume first that \(\mathcal P=\Ctau(\mathbf M)\), and fix \(A\in\mathcal P\).
Let \(\mathcal Q\) be a $\tau$-stable partition of \(\PosId(\mathbf M_A)=A\) computed inside the component.
By Lemma~\ref{lem:component-calculus}\eqref{item:component-calculus-tprod}, the same partition is $\tau$-stable when viewed ambiently in \(\mathbf M\).
Hence \(\mathcal P':=(\mathcal P\setminus\{A\})\cup\mathcal Q\) is a $\tau$-stable partition of \(\PosId(\mathbf M)\).
Since \(\mathcal P\) is the finest $\tau$-stable partition by Theorem~\ref{thm:finest-tau-stable}, the two partitions \(\mathcal P\) and \(\mathcal P'\) refine one another.
Therefore \(\mathcal P'=\mathcal P\), and so \(\mathcal Q=\{A\}\).
Thus \(\mathbf M_A\) is $\tau$-cohesive.

Conversely, assume that every \(\mathbf M_A\) is $\tau$-cohesive.
Since \(\Ctau(\mathbf M)\) is finest $\tau$-stable, it refines \(\mathcal P\).
For \(A\in\mathcal P\), let \(\mathcal Q_A:=\{B\in\Ctau(\mathbf M):B\subseteq A\}\).
By Lemma~\ref{lem:component-calculus}\eqref{item:component-calculus-tprod}, this restriction is a $\tau$-stable partition of \(\PosId(\mathbf M_A)=A\) inside \(\mathbf M_A\).
Cohesiveness of \(\mathbf M_A\) forces \(\mathcal Q_A=\{A\}\).
Thus \(\Ctau(\mathbf M)\) does not split any block of \(\mathcal P\); since it already refines \(\mathcal P\), we get \(\Ctau(\mathbf M)=\mathcal P\).
\end{proof}

\begin{corollary}\label{cor:Ctau-components-are-tau-cohesive}
Every component monoid of the $\tau$-cohesive decomposition is a $\tau$-cohesive finite local-unit-aligned totally ordered monoid.
Equivalently, if $A\in\Ctau(\mathbf M)$, then the corresponding unital subsemigroup $\layer{A}\subseteq M$, endowed with its intrinsic identity \( e_A=\min A \) and with the restricted order and multiplication, is $\tau$-cohesive.
\end{corollary}

\begin{proof}
Apply Theorem~\ref{thm:Ctau-characterized-by-components} to the partition \( \mathcal P=\Ctau(\mathbf M). \)
\end{proof}

\section{The canonical direct-system decomposition}

\begin{definition}
Let $\mathbf M$ be local-unit-aligned and let $A,B\in\Cm(\mathbf M)$.
Since $\Cm(\mathbf M)$ is $\tau$-multiplication-coherent, there exists a unique block $C\in\Cm(\mathbf M)$ such that \(\tprod{A}{B}\subseteq C.\)
We denote this block by \(A\vee B.\)
\end{definition}

\begin{lemma}\label{lem:block-join-min}
Let $A,B\in\Cm(\mathbf M)$, and let \( e_A \), \( e_B \) be the identity elements of the component monoids \( \mathbf M_A \), \( \mathbf M_B, \) respectively.
Equivalently, by Proposition~\ref{prop:tau-stable-gives-subsemigroup}, \( e_A=\min A \), \( e_B=\min B \).

Then the block $A\vee B$ is the unique block of $\Cm(\mathbf M)$ containing \(\max\{e_A,e_B\}.\)
Consequently, $$ A\vee B=B \quad\Longleftrightarrow\quad e_A\le e_B. $$

Define a relation on $\Cm(\mathbf M)$ by
\[
A\le_{\min} B
\quad:\Longleftrightarrow\quad
e_A\le e_B.
\]
Then $A\mapsto e_A$ is injective, $\le_{\min}$ is a total order on $\Cm(\mathbf M)$, and \(A\vee B=\max_{\le_{\min}}\{A,B\}.\)
Equivalently,
\[
A\le_{\min} B
\quad\Longleftrightarrow\quad
A\vee B=B.
\]
\end{lemma}

\begin{proof}
By Proposition~\ref{prop:tau-stable-gives-subsemigroup}, $e_A=\min A\in A$, $e_B=\min B\in B$.
Hence, by Lemma~\ref{lem:basic_tau}\eqref{item:tau-fixes-idempotents}, $e_A\in \layer{A}$ and $e_B\in \layer{B}$.
Therefore $\tau(e_Ae_B)\in \tprod{A}{B}\subseteq A\vee B.$
By Lemma~\ref{lem:idempotents-max}, $e_Ae_B=\max\{e_A,e_B\}.$
Since $e_Ae_B$ is a positive idempotent, Lemma~\ref{lem:basic_tau}\eqref{item:tau-fixes-idempotents} yields $\tau(e_Ae_B)=e_Ae_B=\max\{e_A,e_B\}.$
Thus $A\vee B$ is the block containing $\max\{e_A,e_B\}$.

Now \(A\vee B=B\) holds exactly when the block containing \(\max\{e_A,e_B\}\) is \(B\).
The element \(\max\{e_A,e_B\}\) is either \(e_A\) or \(e_B\).
Since \(e_A\in A\), \(e_B\in B\), and distinct blocks of a partition are disjoint, the block containing \(\max\{e_A,e_B\}\) is \(B\) iff either \(\max\{e_A,e_B\}=e_B\), or \(\max\{e_A,e_B\}=e_A\) and \(A=B\).
In the latter case \(e_A=e_B\), so again \(\max\{e_A,e_B\}=e_B\).
Hence the block containing \(\max\{e_A,e_B\}\) is \(B\) iff \(\max\{e_A,e_B\}=e_B\), that is, iff \(e_A\le e_B\).
This proves $A\vee B=B \Longleftrightarrow e_A\le e_B$.

Next, the map $A\mapsto e_A$ is injective: if $e_A=e_B$, then this common element belongs to $A\cap B$, hence $A=B$ because blocks of a partition are disjoint.

Because the skeleton $\Sk(\mathbf M)$ is totally ordered, the induced order on the set $\{\,e_A:A\in\Cm(\mathbf M)\,\}$ is total.
By injectivity of $A\mapsto e_A$, this transports to a total order $\le_{\min}$ on $\Cm(\mathbf M)$ via $$ A\le_{\min} B :\Longleftrightarrow e_A\le e_B. $$ Finally, from the first part we have $A\vee B=B \Longleftrightarrow e_A\le e_B$ and symmetrically $A\vee B=A \Longleftrightarrow e_B\le e_A.$
Therefore $A\vee B=\max_{\le_{\min}}\{A,B\}.$
The final equivalence $A\le_{\min} B \Longleftrightarrow A\vee B=B$ is just a restatement of the same criterion.
\end{proof}

\begin{definition}[Transition maps]
Let $\mathbf M$ be local-unit-aligned.
For $A\in\Cm(\mathbf M)$, let $\mathbf M_A$ be the component monoid from Theorem~\ref{thm:component-local-unit-aligned-monoid}, and let \( e_A \) denote the identity element of $\mathbf M_A$.
Equivalently, by Proposition~\ref{prop:tau-stable-gives-subsemigroup}, \(e_A=\min A.\)
If \(A\le_{\min} B,\) define \(\rho_{A,B}\colon M_A\to M_B\) by \(\rho_{A,B}(x):=xe_B\).
The family
\[
\mathfrak{D}(\mathbf M)
:=
\Bigl((\mathbf M_A)_{A\in\Cm(\mathbf M)},(\rho_{A,B})_{A\le_{\min} B}\Bigr)
\]
is called the \emph{canonical direct-system decomposition} of $\mathbf M$.
\end{definition}

\begin{proposition}\label{prop:rho-well-defined}
If $A\le_{\min} B$ and $x\in M_A=\layer{A}$, then \(\rho_{A,B}(x)=xe_B\in M_B=\layer{B}.\)
Thus each $\rho_{A,B}$ is well defined.
\end{proposition}

\begin{proof}
Let $x\in\layer{A}$.
Since $e_B$ is the identity element of the component monoid \(\mathbf M_B\), we have $e_B\in M_B=\layer{B}.$
Therefore $\tau(xe_B)\in \tprod{A}{B}\subseteq A\vee B=B.$
Hence $xe_B\in\layer{B}$.
\end{proof}

The next theorem gives the decomposition half of the representation theory.
It shows that every finite local-unit-aligned totally ordered monoid determines a chain-indexed direct system from which both the ambient order and the ambient multiplication can be recovered.
The full representation theorem will only be obtained later, after the converse construction and the rigidity result have been incorporated.

\begin{theorem}\label{thm:canonical-direct-system}
Let $\mathbf M$ be local-unit-aligned, and let
\[
 \mathfrak D(\mathbf M)
 =
 \Bigl((\mathbf M_A)_{A\in\Cm(\mathbf M)},(\rho_{A,B})_{A\le_{\min} B}\Bigr)
\]
be its canonical direct-system decomposition.

Then $\mathfrak D(\mathbf M)$ is a direct system of finite totally ordered monoids over the totally ordered index set $\bigl(\Cm(\mathbf M),\le_{\min}\bigr)$, and it reconstructs the ambient ordered monoid $\mathbf M$.

More precisely, the following hold.

\begin{enumerate}
\item\label{item:direct-id} For every $A\in\Cm(\mathbf M)$, \(\rho_{A,A}=\operatorname{id}_{M_A}.\)
\item\label{item:direct-comp} If $A\le_{\min} B\le_{\min} C$, then \(\rho_{B,C}\circ \rho_{A,B}=\rho_{A,C}.\)
\item\label{item:direct-monoid-hom} If $A\le_{\min} B$, then \(\rho_{A,B}\colon M_A\to M_B\) is an isotone monoid homomorphism.

\item\label{item:direct-disjoint-union} The underlying set of \(M\) is partitioned by the underlying sets of the constituent monoids: \(M=\bigcup_{A\in\Cm(\mathbf M)} \layer{A}.\)
\item\label{item:direct-upward-comparison} If $A\le_{\min} B$, $x\in M_A$, and $y\in M_B$, then
\[
 x\le y
 \quad\Longleftrightarrow\quad
 \rho_{A,B}(x)\le y.
\]
Equivalently, $\rho_{A,B}(x)$ is the least element of $\mathbf M_B$ lying above $x$.

\item\label{item:direct-order-recovery} For arbitrary $x\in M_A$ and $y\in M_B$,
\[
 x\le y
 \quad\Longleftrightarrow\quad
 \begin{cases}
 x\le y \text{ inside }\mathbf M_A, & A=B,\\[2mm]
 \rho_{A,B}(x)\le y \text{ inside }\mathbf M_B, & A<_{\min} B,\\[2mm]
 x<\rho_{B,A}(y) \text{ inside }\mathbf M_A, & B<_{\min} A.
 \end{cases}
\]
Hence the ambient order on $M$ is determined by the constituent ordered monoids and the transition maps $\rho_{A,B}$.

\item\label{item:direct-product-recovery} If $x\in M_A$, $y\in M_B$, and \( C:=\max\nolimits_{\le_{\min}}\{A,B\}=A\vee B, \) then \(xy=\rho_{A,C}(x)\rho_{B,C}(y),\) where the product on the right is computed inside the constituent monoid $\mathbf M_C$.
Hence the ambient multiplication on $M$ is determined by the direct-system data.
\end{enumerate}
\end{theorem}

\begin{proof}
By Corollary~\ref{cor:all-components-local-unit-aligned-monoids}, each constituent \(\mathbf M_A\) is a finite local-unit-aligned totally ordered monoid, hence in particular a finite totally ordered monoid.
By Lemma~\ref{lem:block-join-min}, the relation \(\le_{\min}\) is a total order on \(\Cm(\mathbf M)\).
By Proposition~\ref{prop:rho-well-defined}, every map \(\rho_{A,B}\) is well defined.

For \eqref{item:direct-id}, let \(x\in M_A=\layer{A}\).
Since \(e_A\) is the identity of \(\mathbf M_A\) by Proposition~\ref{prop:tau-stable-gives-subsemigroup}, $\rho_{A,A}(x)=xe_A=x.$

For \eqref{item:direct-comp}, let \(A\le_{\min} B\le_{\min} C\) and \(x\in M_A\).
By Lemma~\ref{lem:block-join-min}, we have \(e_A\le e_B\le e_C\), so Lemma~\ref{lem:idempotents-max} gives $e_Be_C=e_C.$
Hence $(\rho_{B,C}\circ \rho_{A,B})(x) = (xe_B)e_C = x(e_Be_C) = xe_C = \rho_{A,C}(x).$

For \eqref{item:direct-monoid-hom}, let \(A\le_{\min} B\).
Isotonicity is immediate from isotonicity of multiplication in the ambient monoid.
Also \(e_A\le e_B\), so by Lemma~\ref{lem:idempotents-max}, \(\rho_{A,B}(e_A)=e_Ae_B=e_B\).
To prove multiplicativity, let \(x,y\in M_A\).
Since \(\rho_{A,B}(y)=ye_B\in M_B\) by Proposition~\ref{prop:rho-well-defined}, and since \(e_B\) is the identity of the constituent monoid \(\mathbf M_B\), we have \(e_B(ye_B)=ye_B\).
Hence
\[
\rho_{A,B}(x)\rho_{A,B}(y)
=(xe_B)(ye_B)
=x\bigl(e_B(ye_B)\bigr)
=x(ye_B)
=(xy)e_B
=\rho_{A,B}(xy),
\]
where \(xy\in M_A\) since \(M_A\) is a submonoid.
Hence \(\rho_{A,B}\) is an isotone monoid homomorphism.

For \eqref{item:direct-disjoint-union}, recall that \(\Lm(\mathbf M)=\{\layer{A}:A\in\Cm(\mathbf M)\}\) is the \(\tau\)-multiplication-coherent decomposition of \(M\), hence a partition of \(M\).

For \eqref{item:direct-upward-comparison}, let \(A\le_{\min} B\), \(x\in M_A\), and \(y\in M_B\).
If \(x\le y\), then \(xe_B\le ye_B=y\), so \(\rho_{A,B}(x)\le y\).
Conversely, if \(\rho_{A,B}(x)\le y\), then \(x\le xe_B=\rho_{A,B}(x)\le y\), because \(e_B\) is positive.
Thus \(x\le y\) iff \(\rho_{A,B}(x)\le y\).

Also \(x\le \rho_{A,B}(x)\).
If \(z\in M_B\) satisfies \(x\le z\), then by the same equivalence \(\rho_{A,B}(x)\le z\).
So \(\rho_{A,B}(x)\) is the least element of \(\mathbf M_B\) lying above \(x\).

For \eqref{item:direct-order-recovery}, the case \(A=B\) is just the internal order of \(\mathbf M_A\), and the case \(A<_{\min}B\) is exactly \eqref{item:direct-upward-comparison}.
If \(B<_{\min}A\), then applying \eqref{item:direct-upward-comparison} to \(y\in M_B\) and \(x\in M_A\) gives \(y\le x\) iff \(\rho_{B,A}(y)\le x\).
Since \(A\neq B\), the sets \(M_A\) and \(M_B\) are disjoint, so \(x\neq y\).
Because the ambient order is total, \(x\le y\) iff \(\neg(y\le x)\), and since \(x,\rho_{B,A}(y)\in M_A\), this is equivalent to \(x<\rho_{B,A}(y)\).
This proves the last case.

Finally, for \eqref{item:direct-product-recovery}, let \(x\in M_A\), \(y\in M_B\), and \(C:=\max\nolimits_{\le_{\min}}\{A,B\}=A\vee B\).
Then \(A\le_{\min}C\) and \(B\le_{\min}C\), so \(\rho_{A,C}\) and \(\rho_{B,C}\) are defined, and $ \rho_{A,C}(x)\rho_{B,C}(y)=(xe_C)(ye_C)=x\bigl(e_C(ye_C)\bigr)=x(ye_C)=(xy)e_C=xy, $ because \(\tau(xy)\in \tprod{A}{B}\subseteq A\vee B=C\), hence \(xy\in M_C\) and therefore \((xy)e_C=xy\).
This proves the product-recovery formula.
\end{proof}

\begin{corollary}\label{cor:canonical-direct-system-recovers-ambient}
The canonical direct-system decomposition $\mathfrak D(\mathbf M)$ determines $\mathbf M$ up to canonical isomorphism.
Indeed, the underlying set is the disjoint union \(\bigcup_{A\in\Cm(\mathbf M)} |\mathbf M_A|,\) the ambient order is recovered by Theorem~\ref{thm:canonical-direct-system}\eqref{item:direct-order-recovery}, and the ambient multiplication is recovered by Theorem~\ref{thm:canonical-direct-system}\eqref{item:direct-product-recovery}.
\end{corollary}

\begin{remark}\label{rem:canonical-direct-system-tree}
Thus every finite local-unit-aligned totally ordered monoid comes equipped with a canonical direct system whose constituent monoids are exactly the components of the $\tau$-multiplication-coherent decomposition, and from this ordinary direct system one recovers both the ambient chain order and the ambient multiplication.
\end{remark}

Thus Section~6 gives the forward half of the main theorem.
Starting from a finite local-unit-aligned totally ordered monoid, it constructs a chain-indexed direct system and proves that this system reconstructs the original ordered monoid.
Sections~7 and~8 establish the converse construction and the finite rigidity result that turns the representation into a rigid one.

\begin{proposition}[Isomorphism invariance of the canonical $\tau$-decomposition data]\label{prop:isomorphism-invariance}
Let
\[
 \mathbf M=\langle M,\le,\cdot,e\rangle
 \qquad\text{and}\qquad
 \mathbf N=\langle N,\le,\cdot,e\rangle
\]
be finite local-unit-aligned totally ordered monoids, and let \(\varphi\colon \mathbf M\to \mathbf N\) be an isomorphism of totally ordered monoids.
For $A\subseteq \PosId(\mathbf M)$ and for a partition $\mathcal P$ of $\PosId(\mathbf M)$, write
\[
 \varphi[A]:=\{\varphi(u):u\in A\},
 \qquad
 \varphi[\mathcal P]:=\{\varphi[A]:A\in\mathcal P\}.
\]

Then the following hold.

\begin{enumerate}
\item\label{item:iso-posid-tau} \(\varphi[\PosId(\mathbf M)]=\PosId(\mathbf N),\) and for every $x\in M$, \(\varphi(\tau_{\mathbf M}(x))=\tau_{\mathbf N}(\varphi(x)).\)
\item\label{item:iso-layers-products} For every $A,B\subseteq\PosId(\mathbf M)$, \(\varphi[\layer{A}^{\mathbf M}]=\layer{\varphi[A]}^{\mathbf N},\) and
\[
 \varphi\bigl[(\tprod{A}{B})^{\mathbf M}\bigr]
 =
 (\tprod{\varphi[A]}{\varphi[B]})^{\mathbf N}.
\]

\item\label{item:iso-stable-coherent} A partition $\mathcal P$ of $\PosId(\mathbf M)$ is $\tau$-stable if and only if $\varphi[\mathcal P]$ is $\tau$-stable as a partition of $\PosId(\mathbf N)$.
Likewise, $\mathcal P$ is $\tau$-multiplication-coherent if and only if $\varphi[\mathcal P]$ is $\tau$-multiplication-coherent.

\item\label{item:iso-Ctau-Cm}
\[
 \varphi[\Ctau(\mathbf M)]=\Ctau(\mathbf N),
 \qquad
 \varphi[\Cm(\mathbf M)]=\Cm(\mathbf N).
\]
Consequently,
\[
 \varphi[\Ltau(\mathbf M)]=\Ltau(\mathbf N),
 \qquad
 \varphi[\Lm(\mathbf M)]=\Lm(\mathbf N).
\]

\item\label{item:iso-direct-system} Let
\[
 \mathfrak D(\mathbf M)
 =
 \Bigl((\mathbf M_A)_{A\in\Cm(\mathbf M)},(\rho^{\mathbf M}_{A,B})_{A\le_{\min} B}\Bigr)
\]
and
\[
 \mathfrak D(\mathbf N)
 =
 \Bigl((\mathbf N_C)_{C\in\Cm(\mathbf N)},(\rho^{\mathbf N}_{C,D})_{C\le_{\min} D}\Bigr)
\]
be the canonical direct-system decompositions of $\mathbf M$ and $\mathbf N$.
Then the map
\[
 \sigma\colon \Cm(\mathbf M)\to \Cm(\mathbf N),
 \qquad
 \sigma(A):=\varphi[A],
\]
is an order isomorphism between
\[
 \bigl(\Cm(\mathbf M),\le_{\min}\bigr)
 \qquad\text{and}\qquad
 \bigl(\Cm(\mathbf N),\le_{\min}\bigr),
\]
and for every $A\in\Cm(\mathbf M)$ the restriction \(\varphi_A:=\varphi\!\upharpoonright_{M_A}\colon M_A\to N_{\sigma(A)}\) is an isomorphism of finite totally ordered monoids.
Moreover, whenever $A\le_{\min} B$,
\[
 \varphi_B\circ \rho^{\mathbf M}_{A,B}
 =
 \rho^{\mathbf N}_{\sigma(A),\sigma(B)}\circ \varphi_A.
\]
Hence the canonical direct-system decompositions of $\mathbf M$ and $\mathbf N$ are isomorphic.
\end{enumerate}
\end{proposition}

\begin{proof}
Since $\varphi$ is an isomorphism of ordered monoids, it preserves and reflects the unit, multiplication, and order.

For item~\eqref{item:iso-posid-tau}, positivity and idempotence are preserved and reflected, so $\varphi[\PosId(\mathbf M)]=\PosId(\mathbf N).$
Also, for $x,z\in M$, $z\in R_{\mathbf M}(x) \Longleftrightarrow xz\le x \Longleftrightarrow \varphi(x)\varphi(z)\le \varphi(x) \Longleftrightarrow \varphi(z)\in R_{\mathbf N}(\varphi(x)).$
Thus $\varphi[R_{\mathbf M}(x)]=R_{\mathbf N}(\varphi(x))$, and similarly $\varphi[L_{\mathbf M}(x)]=L_{\mathbf N}(\varphi(x)).$
Because \(\varphi\) is an order isomorphism, it preserves maxima, so $\varphi(\tau_{r,\mathbf M}(x))=\tau_{r,\mathbf N}(\varphi(x))$, $\varphi(\tau_{\ell,\mathbf M}(x))=\tau_{\ell,\mathbf N}(\varphi(x)).$
Since both monoids are local-unit-aligned, $\varphi(\tau_{\mathbf M}(x))=\tau_{\mathbf N}(\varphi(x)).$

For item~\eqref{item:iso-layers-products}, let \(A\subseteq\PosId(\mathbf M)\).
Then, by item~\eqref{item:iso-posid-tau}, $y\in \varphi[\layer{A}^{\mathbf M}] \Longleftrightarrow \tau_{\mathbf N}(y)\in \varphi[A] \Longleftrightarrow y\in \layer{\varphi[A]}^{\mathbf N}$, so $\varphi[\layer{A}^{\mathbf M}]=\layer{\varphi[A]}^{\mathbf N}.$
Hence, for \(A,B\subseteq\PosId(\mathbf M)\),
\begin{align*}
\varphi\bigl[(\tprod{A}{B})^{\mathbf M}\bigr]
&=
\{\varphi(\tau_{\mathbf M}(xy)):x\in\layer{A}^{\mathbf M},\,y\in\layer{B}^{\mathbf M}\}\\
&=
\{\tau_{\mathbf N}(\varphi(x)\varphi(y)):x\in\layer{A}^{\mathbf M},\,y\in\layer{B}^{\mathbf M}\}\\
&=
(\tprod{\varphi[A]}{\varphi[B]})^{\mathbf N}.
\end{align*}

For item~\eqref{item:iso-stable-coherent}, item~\eqref{item:iso-layers-products} transports the defining conditions of \(\tau\)-stability and \(\tau\)-multiplication coherence blockwise from \(\mathcal P\) to \(\varphi[\mathcal P]\), and conversely by applying the same argument to \(\varphi^{-1}\).

For item~\eqref{item:iso-Ctau-Cm}, item~\eqref{item:iso-stable-coherent} shows that \(\varphi[\Ctau(\mathbf M)]\) is a \(\tau\)-stable partition of \(\PosId(\mathbf N)\).
Since \(\Ctau(\mathbf N)\) is the finest \(\tau\)-stable partition, $\Ctau(\mathbf N)\ \text{refines}\ \varphi[\Ctau(\mathbf M)].$
Applying the same argument to \(\varphi^{-1}\) yields the reverse refinement, so $\varphi[\Ctau(\mathbf M)]=\Ctau(\mathbf N).$
Exactly the same argument with ``\(\tau\)-stable'' replaced by ``\(\tau\)-multiplication-coherent'' gives $\varphi[\Cm(\mathbf M)]=\Cm(\mathbf N).$
The statements about \(\Ltau\) and \(\Lm\) now follow from item~\eqref{item:iso-layers-products}.

For item~\eqref{item:iso-direct-system}, let $\sigma(A):=\varphi[A].$
By item~\eqref{item:iso-Ctau-Cm}, this is a bijection from \(\Cm(\mathbf M)\) onto \(\Cm(\mathbf N)\).

Let $e_A$ be the identity element of \(\mathbf M_A\).
By Proposition~\ref{prop:tau-stable-gives-subsemigroup}, $e_A=\min A.$
Since \(\varphi\) is an order isomorphism on the skeleton, $\min \sigma(A)=\varphi(e_A).$
Again by Proposition~\ref{prop:tau-stable-gives-subsemigroup}, the identity element of \(\mathbf N_{\sigma(A)}\) is $e_{\sigma(A)}=\min \sigma(A)=\varphi(e_A).$
Hence for \(A,B\in\Cm(\mathbf M)\), $A\le_{\min} B \quad\Longleftrightarrow\quad e_A\le e_B \quad\Longleftrightarrow\quad \varphi(e_A)\le \varphi(e_B) \quad\Longleftrightarrow\quad e_{\sigma(A)}\le e_{\sigma(B)} \quad\Longleftrightarrow\quad \sigma(A)\le_{\min} \sigma(B).$
Thus \(\sigma\) is an order isomorphism.

Now let \(A\in\Cm(\mathbf M)\).
By item~\eqref{item:iso-layers-products}, $\varphi[M_A] = \varphi[\layer{A}^{\mathbf M}] = \layer{\sigma(A)}^{\mathbf N} = \mathbf N_{\sigma(A)}.$
So the restriction $\varphi_A:=\varphi\!\upharpoonright_{M_A}\colon M_A\to N_{\sigma(A)}$ is a well-defined isomorphism of finite totally ordered monoids.

Finally, let \(A\le_{\min} B\) and \(x\in M_A\).
Since $e_{\sigma(B)}=\varphi(e_B)$, we have
\begin{align*}
\varphi_B(\rho^{\mathbf M}_{A,B}(x))
&=
\varphi(xe_B)\\
&=
\varphi(x)\varphi(e_B)\\
&=
\varphi_A(x)\,e_{\sigma(B)}\\
&=
\rho^{\mathbf N}_{\sigma(A),\sigma(B)}(\varphi_A(x)).
\end{align*}
Therefore $\varphi_B\circ \rho^{\mathbf M}_{A,B} = \rho^{\mathbf N}_{\sigma(A),\sigma(B)}\circ \varphi_A.$
Hence the canonical direct-system decompositions of \(\mathbf M\) and \(\mathbf N\) are isomorphic.
\end{proof}

\section{Constructing local-unit-aligned totally ordered monoids from direct systems}

\begin{definition}
Let $(I,\le)$ be a finite chain.
A \emph{direct system of finite local-unit-aligned totally ordered monoids over $I$} consists of:

\begin{enumerate}
\item For each $i\in I$, a finite local-unit-aligned totally ordered monoid \(\mathbf M_i=\langle M_i,\le_i,\cdot_i,e_i\rangle.\)
\item For each $i\le j$, an isotone unital monoid homomorphism \(\rho_{i,j}\colon M_i\to M_j\) such that:
\begin{enumerate}
\item $\rho_{i,i}=\mathrm{id}$,
\item $\rho_{j,k}\circ\rho_{i,j}=\rho_{i,k}$ whenever $i\le j\le k$.
\end{enumerate}
\end{enumerate}
\end{definition}

For canonical systems arising from a monoid, the component universes are automatically disjoint, since they are the distinct fibres $ \layer{A}=\tau^{-1}(A) $ over the blocks \(A\in \Cm(\mathbf M)\).
For an abstract direct system used in the converse construction, we work with tagged disjoint copies when forming the external union.

\begin{definition}[Directed lexicographic order]\label{def:directed-lexicographic-order}
Let $(I,\le)$ be a finite chain, and let \(\bigl((\mathbf M_i)_{i\in I},(\rho_{i,j})_{i\le j}\bigr)\) be a direct system of finite totally ordered monoids, where \(\mathbf M_i=\langle M_i,\le_i,\cdot_i,e_i\rangle.\)
Put \(M:=\bigsqcup_{i\in I} M_i.\)
The \emph{strict directed lexicographic order} on $M$ is the binary relation $<$ defined as follows: for $x\in M_i$ and $y\in M_j$, writing \(m:=\max\{i,j\},\) we set
\[
x<y
\quad\Longleftrightarrow\quad
\bigl(\rho_{i,m}(x)<_m \rho_{j,m}(y)\bigr)
\ \text{or}\
\bigl(\rho_{i,m}(x)=\rho_{j,m}(y)\ \text{and}\ i<j\bigr).
\]

Its reflexive closure is denoted by $\le$.
\end{definition}
Thus two elements are compared after transport to the first common upper level; only if those transported values coincide do we use the index order to break ties.
This is the directed lexicographic order introduced in \cite{Jenei2022GroupRepr}, here specialized to direct systems indexed by finite chains.

\begin{lemma}[The directed lexicographic order is total]\label{lem:dlex-total-order}
Let $(I,\le)$ be a finite chain, and let \(\bigl((\mathbf M_i)_{i\in I},(\rho_{i,j})_{i\le j}\bigr)\) be a direct system of finite totally ordered monoids.
Let \(M:=\bigsqcup_{i\in I} M_i\) and equip \(M\) with the directed lexicographic order from Definition~\ref{def:directed-lexicographic-order}.
Then \(\le\) is a total order on \(M\).
\end{lemma}

\begin{proof}
Let \(x\in M_i\), \(y\in M_j\), and put $m:=\max\{i,j\}.$
Since \(M_m\) is totally ordered, exactly one of the following holds: $\rho_{i,m}(x)<_m \rho_{j,m}(y)$, $\rho_{i,m}(x)=\rho_{j,m}(y)$, $\rho_{j,m}(y)<_m \rho_{i,m}(x)$.
By Definition~\ref{def:directed-lexicographic-order}, in the first case \(x<y\), in the third case \(y<x\), and in the middle case either \(i<j\), \(j<i\), or \(i=j\).
If \(i<j\), then again \(x<y\); if \(j<i\), then \(y<x\); and if \(i=j\), then \(m=i=j\), so \(\rho_{i,m}=\rho_{j,m}=\mathrm{id}\), whence \(x=y\).
Thus for all \(x,y\in M\), exactly one of $x<y$, $ x=y$, $ y<x$ holds.

It remains to prove transitivity of \(<\).
We first record a simple consequence of isotonicity.
If \(a\le b\le c\), \(p\in M_a\), and \(q\in M_b\), then
\[
\rho_{a,c}(p)<_c \rho_{b,c}(q)
\quad\Longrightarrow\quad
\rho_{a,b}(p)<_b q.
\]
Indeed, if \(q\le_b \rho_{a,b}(p)\), then isotonicity of \(\rho_{b,c}\) would give \(\rho_{b,c}(q)\le_c \rho_{b,c}\rho_{a,b}(p)=\rho_{a,c}(p)\), contradicting \(\rho_{a,c}(p)<_c\rho_{b,c}(q)\).

Now let \(x\in M_i\), \(y\in M_j\), and \(z\in M_k\), and suppose \(x<y\) and \(y<z\).
We consider the possible relative positions of \(i,j,k\).

If \(i\le j\le k\), then \(x<y\) gives \(\rho_{i,j}(x)\le_j y\), and \(y<z\) gives \(\rho_{j,k}(y)\le_k z\).
Hence \(\rho_{i,k}(x)\le_k z\), so \(x<z\) if \(i<k\), while if \(i=k\) then \(i=j=k\) and \(x<_i y<_i z\).

If \(i\le k\le j\), then \(x<y\) gives \(\rho_{i,j}(x)\le_j y\), while \(y<z\) gives \(y<_j\rho_{k,j}(z)\).
Thus \(\rho_{i,j}(x)<_j\rho_{k,j}(z)\).
By the observation above, \(\rho_{i,k}(x)<_k z\), so \(x<z\).

If \(j\le i\le k\), then \(x<y\) gives $ x<_i\rho_{j,i}(y), $ and \(y<z\) gives $ \rho_{j,k}(y)\le_k z. $
By isotonicity of \(\rho_{i,k}\), the first inequality yields $ \rho_{i,k}(x)\le_k \rho_{i,k}(\rho_{j,i}(y)) =\rho_{j,k}(y). $
Hence $ \rho_{i,k}(x)\le_k \rho_{j,k}(y)\le_k z. $
If \(i<k\), then the directed lexicographic definition gives \(x<z\) from \(\rho_{i,k}(x)\le_k z\), since the lower index \(i\) wins ties.
If \(i=k\), then \(x<_i\rho_{j,i}(y)\le_i z\), so \(x<_i z\), and again \(x<z\).

If \(j\le k\le i\), then \(x<y\) gives \(x<_i\rho_{j,i}(y)\), and \(y<z\) gives \(\rho_{j,k}(y)\le_k z\).
Transporting the latter inequality to level \(i\), we get \(\rho_{j,i}(y)\le_i\rho_{k,i}(z)\).
Therefore \(x<_i\rho_{k,i}(z)\), so \(x<z\).

If \(k\le i\le j\), then \(x<y\) gives \(\rho_{i,j}(x)\le_j y\), while \(y<z\) gives \(y<_j\rho_{k,j}(z)\).
Thus $ \rho_{i,j}(x)<_j\rho_{k,j}(z) = \rho_{i,j}(\rho_{k,i}(z)). $
For an isotone map \(f\) between chains, \(f(a)<f(b)\) implies \(a<b\); otherwise \(b\le a\), and isotonicity would give \(f(b)\le f(a)\).
Applying this to \(f=\rho_{i,j}\), we obtain $ x<_i\rho_{k,i}(z). $
Hence \(x<z\).

Finally, if \(k\le j\le i\), then \(x<y\) gives $ x<_i\rho_{j,i}(y), $ and \(y<z\) gives $ y<_j\rho_{k,j}(z). $
Applying the isotone map \(\rho_{j,i}\) to the latter inequality gives $ \rho_{j,i}(y)\le_i \rho_{j,i}(\rho_{k,j}(z)) =\rho_{k,i}(z). $
Combining it with \(x<_i\rho_{j,i}(y)\), we obtain $ x<_i\rho_{k,i}(z). $
Therefore \(x<z\).

Thus \(<\) is transitive.

Therefore the reflexive closure \(\le\) of \(<\) is a total order on \(M\).
\end{proof}

\begin{remark}[Working forms of the directed lexicographic order]
\label{rem:directed-lex-order-working-forms}
Let \(x\in M_i\) and \(y\in M_j\).

If \(i\le j\), then \(x\le y \iff \rho_{i,j}(x)\le_j y.\)
If \(j<i\), then \(x\le y \iff x<_i \rho_{j,i}(y).\)
Equivalently, for the strict order, \(i<j \Longrightarrow x<y \iff \rho_{i,j}(x)\le_j y,\) and \(j<i \Longrightarrow x<y \iff x<_i \rho_{j,i}(y).\)
These are immediate consequences of Definition~\ref{def:directed-lexicographic-order}.
\end{remark}

\begin{definition}[Strict compatibility]\label{def:strict-compatibility}
Let \((I,\le)\) be a finite chain, and let \(\bigl((\mathbf M_i)_{i\in I},(\rho_{i,j})_{i\le j}\bigr)\) be a direct system of finite totally ordered monoids, where \(\mathbf M_i=\langle M_i,\le_i,\cdot_i,e_i\rangle.\)
We say that the system is \emph{strictly compatible} if for all indices \(i<j\) and \(k<j\), and for all \(u\in M_j\), \(x\in M_i\), and \(y\in M_k\), the following hold:
\begin{enumerate}
\taggeditem{\textup{(R)}}{item:strict-compatibility-right}
\makebox[\linewidth][c]{%
\(\displaystyle
u<_j \rho_{i,j}(x)
\quad\Longrightarrow\quad
u\cdot_j \rho_{k,j}(y)
<
\rho_{i,j}(x)\cdot_j \rho_{k,j}(y);
\)}

\taggeditem{\textup{(L)}}{item:strict-compatibility-left}
\makebox[\linewidth][c]{%
\(\displaystyle
u<_j \rho_{i,j}(x)
\quad\Longrightarrow\quad
\rho_{k,j}(y)\cdot_j u
<
\rho_{k,j}(y)\cdot_j \rho_{i,j}(x).
\)}
\end{enumerate}

An isotone unital monoid homomorphism between finite totally ordered monoids is called \emph{strictly compatible} if, when regarded as the unique proper transition map of a direct system over a two-element chain, that direct system is strictly compatible.
\end{definition}

After the same-level and forward cross-level cases are reduced to isotonicity inside a common upper component, the remaining possible obstruction to isotonicity of the multiplication on \( M=\bigsqcup_{i\in I} M_i \) with respect to the directed lexicographic order occurs in the reversed cross-level situation: an element in a higher level is compared with the image of an element from a lower level, and one then multiplies by an element coming from another lower level.
The next lemma shows that strict compatibility is exactly the condition needed to control this case.

\begin{lemma}[Exact isotonicity criterion for finite chain-indexed systems]
\label{lem:exact-finite-chain-criterion}
Let $(I,\le)$ be a finite chain, let \(\bigl((\mathbf M_i)_{i\in I},(\rho_{i,j})_{i\le j}\bigr)\) be a direct system of finite totally ordered monoids, and form the external union below using the tagged-disjoint-copy convention above: \(M:=\bigsqcup_{i\in I} M_i.\)
In this lemma, the unsubscripted symbols \(\le\) and \(<\) denote the ambient directed lexicographic order on \(M\), while \(\le_i\) and \(<_i\) always denote the internal order of \(\mathbf M_i\).

Equip \(M\) with the directed lexicographic order: for \(x\in M_i\) and \(y\in M_j\), writing \(m:=\max\{i,j\}\),
\[
x\le y
\quad\Longleftrightarrow\quad
\bigl(\rho_{i,m}(x)<_m \rho_{j,m}(y)\bigr)
\ \text{or}\
\bigl(\rho_{i,m}(x)=\rho_{j,m}(y)\ \text{and}\ i\le j\bigr).
\]

Define multiplication on \(M\) by
\[
x\in M_i,\ y\in M_j,\ m:=\max\{i,j\},
\qquad
xy:=\rho_{i,m}(x)\cdot_m \rho_{j,m}(y).
\]

Then this multiplication is isotone in both arguments with respect to the ambient directed lexicographic order \(\le\) if and only if the direct system is strictly compatible in the sense of Definition~\ref{def:strict-compatibility}.
\end{lemma}

\begin{proof}
We use the standard order equivalences for the directed lexicographic order: for \(x\in M_i\) and \(y\in M_j\),
\begin{equation}\label{eq:dlex-upper-proof}
i\le j
\quad\Longrightarrow\quad
x\le y \iff \rho_{i,j}(x)\le_j y,
\end{equation}
and
\begin{equation}\label{eq:dlex-lower-proof}
j<i
\quad\Longrightarrow\quad
x\le y \iff x<_i \rho_{j,i}(y).
\end{equation}

\medskip

\noindent\textit{Necessity.}
Assume that the multiplication is isotone in both arguments.

We prove \ref{item:strict-compatibility-right}.
Let \(i<j\), \(k<j\), \(u\in M_j\), \(x\in M_i\), \(y\in M_k\), and suppose that \(u<_j \rho_{i,j}(x)\).
By \eqref{eq:dlex-lower-proof}, this is equivalent to \(u\le x\) in the ambient order on \(M\).
Right isotonicity yields \(uy\le xy\).
Here \(uy\in M_j\), whereas \(xy\in M_m\) with \(m:=\max\{i,k\}<j\).
Applying \eqref{eq:dlex-lower-proof} once more, we obtain \(uy<_j \rho_{m,j}(xy).\)
Now
\[
uy=u\cdot_j \rho_{k,j}(y),
\qquad
\rho_{m,j}(xy)
=
\rho_{i,j}(x)\cdot_j \rho_{k,j}(y),
\]
so \(u\cdot_j \rho_{k,j}(y) < \rho_{i,j}(x)\cdot_j \rho_{k,j}(y),\) which is exactly \ref{item:strict-compatibility-right}.

We now prove \ref{item:strict-compatibility-left}.
Let \(i<j\), \(k<j\), \(u\in M_j\), \(x\in M_i\), \(y\in M_k\), and suppose that \(u<_j\rho_{i,j}(x)\).
Again, by \eqref{eq:dlex-lower-proof}, this is equivalent to \(u\le x\) in the ambient order on \(M\).
Left isotonicity yields \(yu\le yx\).
Here \(yu\in M_j\), whereas \(yx\in M_m\) with \(m:=\max\{i,k\}<j\).
Applying \eqref{eq:dlex-lower-proof} once more, we obtain \(yu<_j\rho_{m,j}(yx).\)
Now
\[
yu=\rho_{k,j}(y)\cdot_j u,
\qquad
\rho_{m,j}(yx)
=
\rho_{k,j}(y)\cdot_j\rho_{i,j}(x),
\]
so \(\rho_{k,j}(y)\cdot_j u < \rho_{k,j}(y)\cdot_j\rho_{i,j}(x),\) which is exactly \ref{item:strict-compatibility-left}.
Hence the system is strictly compatible.

\medskip

\noindent\textit{Sufficiency.}
Assume now that the system is strictly compatible.
We first prove right isotonicity.
Let \(p\in M_i\), \(q\in M_j\), \(r\in M_k\), and suppose that \(p\le q\).
We must show that \(pr\le qr\).

If \(i=j\), then both products lie in \(M_m\) for \(m:=\max\{i,k\}\), and
\[
pr=\rho_{i,m}(p)\cdot_m \rho_{k,m}(r),
\qquad
qr=\rho_{i,m}(q)\cdot_m \rho_{k,m}(r).
\]
Since \(\rho_{i,m}\) is isotone and multiplication in \(\mathbf M_m\) is isotone, we get \(pr\le qr\).

Assume next that \(i<j\).
Then by \eqref{eq:dlex-upper-proof}, \(\rho_{i,j}(p)\le_j q.\)
If \(k\le j\), then \(qr\in M_j\) and \(pr\in M_m\) with \(m:=\max\{i,k\}\le j\).
Thus it is enough to prove \(\rho_{m,j}(pr)\le_j qr\).
But
\[
\rho_{m,j}(pr)=\rho_{i,j}(p)\cdot_j \rho_{k,j}(r),
\qquad
qr=q\cdot_j \rho_{k,j}(r),
\]
so the claim follows from isotonicity in \(\mathbf M_j\).
If \(j<k\), then both products lie in \(M_k\), and applying \(\rho_{j,k}\) to \(\rho_{i,j}(p)\le_j q\) gives \(\rho_{i,k}(p)\le_k \rho_{j,k}(q),\) hence \(pr\le qr\) by isotonicity in \(\mathbf M_k\).

Finally, assume that \(j<i\).
Then by \eqref{eq:dlex-lower-proof}, \(p<_i \rho_{j,i}(q).\)
If \(k<i\), then \(pr\in M_i\), while \(qr\in M_m\) with \(m:=\max\{j,k\}<i\).
Therefore it suffices to prove \(pr<_i \rho_{m,i}(qr).\)
Now
\[
pr=p\cdot_i \rho_{k,i}(r),
\qquad
\rho_{m,i}(qr)=\rho_{j,i}(q)\cdot_i \rho_{k,i}(r),
\]
and strict compatibility \ref{item:strict-compatibility-right} yields \(p\cdot_i \rho_{k,i}(r) < \rho_{j,i}(q)\cdot_i \rho_{k,i}(r),\) so indeed \(pr\le qr\).
If \(i\le k\), then both products lie in \(M_k\), and transporting \(p<_i \rho_{j,i}(q)\) to level \(k\) gives \(\rho_{i,k}(p)\le_k \rho_{j,k}(q),\) whence again \(pr\le qr\).

Thus multiplication is isotone in the right argument.

For left isotonicity, let again \(p\in M_i\), \(q\in M_j\), \(r\in M_k\), and suppose that \(p\le q\).
We prove \(rp\le rq\).

If \(i=j\), then both products lie in \(M_m\), where \(m:=\max\{i,k\}\), and
\[
rp=\rho_{k,m}(r)\cdot_m\rho_{i,m}(p),
\qquad
rq=\rho_{k,m}(r)\cdot_m\rho_{i,m}(q).
\]
Since \(\rho_{i,m}\) is isotone and multiplication in \(\mathbf M_m\) is isotone, \(rp\le rq\).

Assume next that \(i<j\).
Then \(\rho_{i,j}(p)\le_j q\).
If \(k\le j\), put \(m:=\max\{i,k\}\).
It is enough to prove \(\rho_{m,j}(rp)\le_j rq\), and indeed
\[
\rho_{m,j}(rp)=\rho_{k,j}(r)\cdot_j\rho_{i,j}(p),
\qquad
rq=\rho_{k,j}(r)\cdot_j q.
\]
The claim follows from isotonicity in \(\mathbf M_j\).
If \(j<k\), then both products lie in \(M_k\).
Applying \(\rho_{j,k}\) to \(\rho_{i,j}(p)\le_j q\) gives \(\rho_{i,k}(p)\le_k\rho_{j,k}(q),\) and isotonicity in \(\mathbf M_k\) gives \(rp\le rq\).

Finally, assume that \(j<i\).
Then \(p<_i\rho_{j,i}(q)\).
If \(k<i\), then \(rp\in M_i\), while \(rq\in M_m\) with \(m:=\max\{j,k\}<i\).
It is enough to prove \(rp<_i\rho_{m,i}(rq).\)
Now
\[
rp=\rho_{k,i}(r)\cdot_i p,
\qquad
\rho_{m,i}(rq)=\rho_{k,i}(r)\cdot_i\rho_{j,i}(q),
\]
and strict compatibility \ref{item:strict-compatibility-left} yields \(\rho_{k,i}(r)\cdot_i p < \rho_{k,i}(r)\cdot_i\rho_{j,i}(q).\)
Thus \(rp\le rq\).
If \(i\le k\), then both products lie in \(M_k\).
Transporting \(p<_i\rho_{j,i}(q)\) to level \(k\) gives \(\rho_{i,k}(p)\le_k\rho_{j,k}(q),\) and isotonicity in \(\mathbf M_k\) again gives \(rp\le rq\).

Hence multiplication is isotone in the left argument as well.

Therefore the multiplication on \(M\) is isotone in both arguments if and only if the direct system is strictly compatible.
\end{proof}

\begin{theorem}\label{thm:construction-from-direct-system}
Let $(I,\le)$ be a finite chain, let \(\bigl((\mathbf M_i)_{i\in I},(\rho_{i,j})_{i\le j}\bigr)\) be a direct system of finite local-unit-aligned totally ordered monoids, where \(\mathbf M_i=\langle M_i,\le_i,\cdot_i,e_i\rangle,\) and assume that the direct system is strictly compatible in the sense of Definition~\ref{def:strict-compatibility}.

Let \(i_0:=\min I\), and define \(M:=\bigsqcup_{i\in I} M_i.\)
Then there exists a unique structure of a finite local-unit-aligned totally ordered monoid \(\mathbf M=\langle M,\le,\cdot,e\rangle\) on \(M\), where the global identity element is \(e:=e_{i_0},\) such that:

\begin{enumerate}
\taggeditem{\textup{(1)}}{item:construction-subsemigroup}
Each \(\mathbf M_i\) is an ordered subsemigroup of \(\mathbf M\); with the restricted multiplication, it is a monoid in its own right with identity \(e_i\).

\taggeditem{\textup{(2)}}{item:construction-order-condition}
The order \(\le\) is the directed lexicographic order induced by the direct system in the sense of Definition~\ref{def:directed-lexicographic-order}.
Equivalently, for all \(x\in M_i\) and \(y\in M_j\) with \(i\le j\),
\begin{equation}\label{eq:construction-order}
x\le y \iff \rho_{i,j}(x)\le_j y.
\end{equation}

\taggeditem{\textup{(3)}}{item:construction-multiplication}
The multiplication is given by:
\begin{equation}\label{defMULT}
x\in M_i,\ y\in M_j,\ k:=\max\{i,j\},
\qquad
xy:=\rho_{i,k}(x)\cdot_k \rho_{j,k}(y).
\end{equation}

\taggeditem{\textup{(4)}}{item:construction-local-units}
For every \(i\in I\) and every \(x\in M_i\), the ambient local unit of \(x\) coincides with its internal local unit in \(\mathbf M_i\):
\begin{equation}\label{eq:construction-local-units-preserved}
\tau_{\mathbf M}(x)=\tau_{\mathbf M_i}(x).
\end{equation}
In particular, \(\tau_{\mathbf M}[M_i]\subseteq M_i\) for every \(i\in I\).
\end{enumerate}

With these operations, \(\mathbf M\) is a finite local-unit-aligned totally ordered monoid.
\end{theorem}

\begin{proof}
Put \( M:=\bigsqcup_{i\in I} M_i \) and \( e:=e_{i_0}. \)

\medskip

\noindent
\textit{Step 1: Definition of the order and multiplication.}
Equip \(M\) with the directed lexicographic order from
Definition~\ref{def:directed-lexicographic-order}; denote again its strict part by \(<\)
and its reflexive closure by \(\le\).
Define multiplication on \(M\) by \eqref{defMULT}.

By Remark~\ref{rem:directed-lex-order-working-forms}, this order satisfies \(x\le y \iff \rho_{i,j}(x)\le_j y\) for \(x\in M_i\), \(y\in M_j\), and \(i\le j\), which is exactly \eqref{eq:construction-order}.

\medskip

\noindent
\textit{Step 2: \(\le\) is a total order.}
By Remark~\ref{rem:directed-lex-order-working-forms} and Lemma~\ref{lem:dlex-total-order}, the directed lexicographic relation is a total order extending each \(\le_i\).
Since \(M\) is a finite disjoint union of finite chains, \(\langle M,\le\rangle\) is a finite chain.

\medskip

\noindent
\textit{Step 3: The order and multiplication restrict correctly on each level.}
If \(x,y\in M_i\), then \(\max\{i,i\}=i\), so \eqref{defMULT} gives
$xy=x\cdot_i y.$
Likewise, since \(i=i\), the directed lexicographic order restricts to the original order:
$x\le y \iff x\le_i y.$
Moreover, for \(x\in M_i\), we have
\(e_i x=e_i\cdot_i x=x\) and \(x e_i=x\cdot_i e_i=x\).
Thus each \(\mathbf M_i\) is an ordered subsemigroup of \(\mathbf M\); with the restricted multiplication, it is a monoid in its own right with identity \(e_i\).

\medskip
	
\noindent
\textit{Step 4: Associativity.}
Let \(x\in M_i\), \(y\in M_j\), \(z\in M_k\), and put
\(m:=\max\{i,j,k\}\). For any \(a,b\le m\), \(u\in M_a\), and \(v\in M_b\),
we have
\(\rho_{\max\{a,b\},m}(uv)=\rho_{a,m}(u)\cdot_m \rho_{b,m}(v).\)
Indeed, this follows immediately from \eqref{defMULT} by applying the
homomorphism \(\rho_{\max\{a,b\},m}\). Hence
\[
\rho_{\max\{i,j\},m}(xy)=\rho_{i,m}(x)\cdot_m \rho_{j,m}(y),
\qquad
\rho_{\max\{j,k\},m}(yz)=\rho_{j,m}(y)\cdot_m \rho_{k,m}(z).
\]
Therefore \((xy)z = \bigl(\rho_{i,m}(x)\cdot_m \rho_{j,m}(y)\bigr)\cdot_m \rho_{k,m}(z),\) and \(x(yz) = \rho_{i,m}(x)\cdot_m \bigl(\rho_{j,m}(y)\cdot_m \rho_{k,m}(z)\bigr).\)
By associativity in \(\mathbf M_m\), it follows that \((xy)z=x(yz).\)

\medskip

\noindent
\textit{Step 5: Global identity.}
Let \(e:=e_{i_0}\). For \(x\in M_i\), we have
\(ex=\rho_{i_0,i}(e)\cdot_i x=e_i\cdot_i x=x\) and
\(xe=x\cdot_i e_i=x\). Thus \(e\) is a global identity.

\medskip

\noindent
\textit{Step 6: Isotonicity.}
By Lemma~\ref{lem:exact-finite-chain-criterion}, the multiplication defined by
\eqref{defMULT} is isotone in both arguments with respect to the directed
lexicographic order if and only if the direct system is strictly compatible.
Since strict compatibility is assumed, multiplication is isotone in both arguments.

\medskip

\noindent
\textit{Step 7: Local-unit alignment.}
To prove local-unit alignment in the ambient monoid, we show that the internal local unit of an element remains maximal among all ambient right and left local units.
In this step,
unsubscripted comparisons are ambient comparisons in \(\mathbf M\), while
subscripted comparisons are internal comparisons in the indicated component.

Let \(x\in M_i\) and put \(u:=\tau_{\mathbf M_i}(x)\).
Then, inside \(\mathbf M_i\), \(x\cdot_i u=x=u\cdot_i x\).
Also, \(e_i\le_i u\), and since \(e_i=\rho_{i_0,i}(e)\), relation \eqref{eq:construction-order} gives \(e\le u\) in the ambient order.
Thus \(u\) is positive in \(\mathbf M\).

We claim that \(u=\tau_{\mathbf M}(x)\).

First, since \(xu=x\) and \(ux=x\) as products computed inside \(\mathbf M_i\), we have \(xu\le x\) and \(ux\le x\) in the ambient order.
Thus \(u\in R_{\mathbf M}(x)\cap L_{\mathbf M}(x)\).

\smallskip

Now let \(z\in M_j\) satisfy \(xz\le x\) in the ambient order.
We show that \(z\le u\) ambiently.

\smallskip

\emph{Case \(j\le i\).}
Then \(xz=x\cdot_i \rho_{j,i}(z)\in M_i\). Since both \(xz\) and \(x\) lie
in \(M_i\), the ambient inequality \(xz\le x\) is just the internal
inequality in \(\mathbf M_i\). Hence
\(x\cdot_i \rho_{j,i}(z)\le_i x\). By the definition of
\(u=\tau_{\mathbf M_i}(x)\), this implies \(\rho_{j,i}(z)\le_i u\).
Therefore \(z\le u\) in the ambient order by \eqref{eq:construction-order}.

\smallskip

\emph{Case \(i<j\).}
Put
$w:=\rho_{i,j}(x)$,
$v:=\rho_{i,j}(u).$
Since \(\rho_{i,j}\) is a monoid homomorphism and \(xu=x\) in \(\mathbf M_i\), we get
$wv=\rho_{i,j}(x)\rho_{i,j}(u)=\rho_{i,j}(xu)=\rho_{i,j}(x)=w.$
Also,
$xz=w\cdot_j z=wz\in M_j.$
Since \(M=\bigsqcup_i M_i\) is the external tagged disjoint union,
equality remembers the summand. Hence \(xz\in M_j\) and \(x\in M_i\)
with \(i<j\) imply \(xz\ne x\).
Hence the ambient inequality \(xz\le x\) is in fact strict, so \(xz<x\)
ambiently. By Remark~\ref{rem:directed-lex-order-working-forms},
\(wz=xz<_j\rho_{i,j}(x)=w\). If \(v\le_j z\), then isotonicity in
\(\mathbf M_j\) gives \(w=wv\le_j wz\), contradicting \(wz<_j w\).
Therefore \(z<_j v=\rho_{i,j}(u)\). Again by
Remark~\ref{rem:directed-lex-order-working-forms}, this implies \(z<u\)
ambiently, hence \(z\le u\) ambiently.

Thus every \(z\in R_{\mathbf M}(x)\) satisfies \(z\le u\), so \(u\) is the greatest element of \(R_{\mathbf M}(x)\).

\smallskip

It remains to check the left local units.
Let \(z\in M_j\) satisfy \(zx\le x\) in the ambient order.
We again prove \(z\le u\).

If \(j\le i\), then \(zx=\rho_{j,i}(z)\cdot_i x\in M_i\), so the ambient inequality \(zx\le x\) is the internal inequality \(\rho_{j,i}(z)\cdot_i x\le_i x\).
Since \(u=\tau_{\mathbf M_i}(x)\), internal left maximality gives \(\rho_{j,i}(z)\le_i u\), hence \(z\le u\) ambiently by \eqref{eq:construction-order}.

If \(i<j\), put \(w:=\rho_{i,j}(x)\) and \(v:=\rho_{i,j}(u)\).
Since \(ux=x\) in \(\mathbf M_i\), we have \(vw=\rho_{i,j}(u)\rho_{i,j}(x)=\rho_{i,j}(ux)=\rho_{i,j}(x)=w.\)
Also \(zx=z\cdot_j w\in M_j\).
Because \(zx\in M_j\), \(x\in M_i\), and \(i<j\), the ambient inequality \(zx\le x\) is strict.
Hence, by Remark~\ref{rem:directed-lex-order-working-forms}, \(zw<_j w.\)
If \(v\le_j z\), isotonicity in \(\mathbf M_j\) gives \(w=vw\le_j zw\), contradicting \(zw<_j w\).
Therefore \(z<_j v=\rho_{i,j}(u)\), and again Remark~\ref{rem:directed-lex-order-working-forms} gives \(z<u\) ambiently.
Thus \(z\le u\).

Hence every \(z\in L_{\mathbf M}(x)\) satisfies \(z\le u\), so \(u\) is also the greatest element of \(L_{\mathbf M}(x)\).

Therefore the greatest right and left local units of \(x\) in the ambient monoid coincide and equal \(u\).
Hence $\tau_{\mathbf M}(x)=u=\tau_{\mathbf M_i}(x).$
Since \(x\in M_i\) was arbitrary, \(\mathbf M\) is local-unit-aligned.

\medskip

\noindent
\textit{Step 8: Uniqueness.}
Condition \eqref{defMULT} determines multiplication uniquely. Condition~\eqref{eq:construction-order}
determines the ambient order uniquely as well: on each level \(M_i\) it is
\(\le_i\), and for \(i\le j\) the cross-level comparison is fixed by
\eqref{eq:construction-order}. If \(j<i\), then the comparison between
\(x\in M_i\) and \(y\in M_j\) is forced by applying the already fixed
\(j\le i\) comparison to \(y\) and \(x\), together with totality and
antisymmetry. Hence every same-level and cross-level comparison is fixed, so
the resulting ordered monoid structure on \(M\) is unique.

\medskip

Therefore \(\mathbf M=\langle M,\le,\cdot,e\rangle\) is a finite local-unit-aligned totally ordered monoid satisfying \ref{item:construction-subsemigroup}--\ref{item:construction-local-units}.
\end{proof}

\begin{example}\label{ex:strict-compatibility-necessary}
The necessity of strict compatibility is explained conceptually in Proposition~\ref{prop:canonical-strict-compatibility}. The following 
example only shows concretely how isotonicity can fail when that condition is omitted.

Let
\[
\mathbf A=\langle A,\le,\cdot,e_A\rangle,
\qquad
\mathbf B=\langle B,\le,\cdot,e_B\rangle
\]
be the two \(2\)-element integral totally ordered monoids \(A=\{a_0<a_1\}\) and \(B=\{b_0<b_1\}\), both with minimum multiplication.
Thus \(e_A=a_1\) and \(e_B=b_1\).
Define
\[
\rho_{A,B}\colon \mathbf A\to \mathbf B,
\qquad
\rho_{A,B}(a_i)=b_i
\quad(i=0,1).
\]

Apply the order and multiplication formulas \eqref{eq:construction-order} and \eqref{defMULT} from Theorem~\ref{thm:construction-from-direct-system} over the chain \(A<B\).
Then the ambient order is \(a_0<b_0<a_1<b_1.\)
Moreover, for \(x\in A\) and \(y\in B\), \(xy=\rho_{A,B}(x)\cdot_B y\) and \(yx=y\cdot_B \rho_{A,B}(x)\).
Now \(b_0<a_1,\) but multiplying on the right by \(a_0\) gives
\[
b_0a_0
=
b_0\cdot_B \rho_{A,B}(a_0)
=
b_0\cdot_B b_0
=
b_0,
\qquad
a_1a_0=a_0.
\]
Since \(a_0<b_0\), we get \(b_0a_0=b_0\nleq a_0=a_1a_0.\)
Thus the constructed multiplication is not isotone.

Equivalently, the strict condition fails already for \(x=b_0\in B\), \(x'=a_1\in A\), and \(y=a_0\in A\), because \(x<\rho_{A,B}(x')=b_1,\) but
\[
x\cdot_B \rho_{A,B}(y)
=
b_0\cdot_B b_0
=
b_0
=
b_1\cdot_B b_0
=
\rho_{A,B}(x')\cdot_B \rho_{A,B}(y).
\]
\end{example}

The construction theorem produces an ambient monoid from prescribed component levels, but it does not yet identify how those levels relate to the intrinsic canonical decomposition of the resulting monoid.
The next proposition supplies this link: the positive idempotents contributed by the construction levels form a $\tau$-multiplication-coherent partition, so the canonical partition refines, but cannot cross, the prescribed levels.

\begin{proposition}\label{prop:construction-level-partition-coherent}
Let $(I,\le)$ be a finite chain, let \(\bigl((\mathbf M_i)_{i\in I},(\rho_{i,j})_{i\le j}\bigr)\) be a direct system of finite local-unit-aligned totally ordered monoids satisfying the hypotheses of Theorem~\ref{thm:construction-from-direct-system}, and let \(\mathbf M=\langle M,\le,\cdot,e\rangle\) be the monoid obtained there, where \(M=\bigsqcup_{i\in I} M_i.\)
For each $i\in I$, put \(P_i:=\PosId(\mathbf M)\cap M_i.\)
Then the following hold.

\begin{enumerate}
\item\label{item:level-partition-posid} Each $P_i$ is nonempty, and \(\mathcal P_{\mathrm{lev}}:=\{P_i:i\in I\}\) is a partition of $\PosId(\mathbf M)$.

\item\label{item:level-partition-internal-posid} For every \(i\in I\), \(P_i=\PosId(\mathbf M_i)\).

\item\label{item:level-partition-layers} For every $i\in I$, \(\layer{P_i}=M_i.\)
\item\label{item:level-partition-coherent} For all $i,j\in I$, if \(k:=\max\{i,j\},\) then \(\tprod{P_i}{P_j}\subseteq P_k.\)
Hence $\mathcal P_{\mathrm{lev}}$ is $\tau$-multiplication-coherent.

\item\label{item:level-partition-canonical-refines} The canonical $\tau$-multiplication-coherent partition $\Cm(\mathbf M)$ refines $\mathcal P_{\mathrm{lev}}$.
Equivalently, every block $A\in\Cm(\mathbf M)$ is contained in a unique $P_i$, and every component \(\layer{A}\in\Lm(\mathbf M)\) is contained in the corresponding construction component $M_i$.
\end{enumerate}
\end{proposition}

\begin{proof}
For \eqref{item:level-partition-posid}, Theorem~\ref{thm:construction-from-direct-system}\ref{item:construction-subsemigroup} gives that each \(\mathbf M_i\) is a monoid in its own right with identity \(e_i\).
Thus \(e_i\) is idempotent.
Moreover, since \(e=e_{i_0}\) and the transition maps preserve identities, \(e_i=\rho_{i_0,i}(e)\).
Hence relation \eqref{eq:construction-order} gives \(e\le e_i\) in the ambient order of \(\mathbf M\).
Thus \(e_i\) is a positive idempotent of \(\mathbf M\), so \(e_i\in P_i\), and therefore each \(P_i\) is nonempty.

Now let $u\in\PosId(\mathbf M)$.
Since $M=\bigsqcup_{i\in I} M_i$, there is a unique $i\in I$ such that $u\in M_i$.
Hence $u\in P_i$.
Thus the sets $P_i$ cover $\PosId(\mathbf M)$, and they are pairwise disjoint because the sets $M_i$ are pairwise disjoint.
Therefore $\mathcal P_{\mathrm{lev}}=\{P_i:i\in I\}$ is a partition of $\PosId(\mathbf M)$.

We next prove \eqref{item:level-partition-internal-posid}.
Fix \(i\in I\).
If \(u\in P_i\), then \(u\in M_i\) and \(u\in\PosId(\mathbf M)\).
By \eqref{eq:construction-local-units-preserved}, \(\tau_{\mathbf M_i}(u)=\tau_{\mathbf M}(u)=u\).
Lemma~\ref{lem:basic_tau} \eqref{item:positive} and \eqref{item:tau-idempotent}, applied inside \(\mathbf M_i\), show that \(\tau_{\mathbf M_i}(u)\) is a positive idempotent of \(\mathbf M_i\).
Since \(\tau_{\mathbf M_i}(u)=u\), it follows that \(u\in\PosId(\mathbf M_i)\).

Conversely, if \(u\in\PosId(\mathbf M_i)\), then \(u\in M_i\), and Lemma~\ref{lem:basic_tau}\eqref{item:tau-fixes-idempotents}, applied inside \(\mathbf M_i\), gives \(\tau_{\mathbf M_i}(u)=u\).
Again by \eqref{eq:construction-local-units-preserved}, \(\tau_{\mathbf M}(u)=u\).
Lemma~\ref{lem:basic_tau}\eqref{item:positive} and \eqref{item:tau-idempotent}, now applied inside \(\mathbf M\), show that \(\tau_{\mathbf M}(u)\) is a positive idempotent of \(\mathbf M\).
Since \(\tau_{\mathbf M}(u)=u\), we get \(u\in\PosId(\mathbf M)\cap M_i=P_i\).
Hence \(P_i=\PosId(\mathbf M_i)\).

For \eqref{item:level-partition-layers}, let \(x\in M_i\).
By \eqref{eq:construction-local-units-preserved}, the ambient local unit of \(x\) is the same as its internal local unit in \(\mathbf M_i\): \(\tau_{\mathbf M}(x)=\tau_{\mathbf M_i}(x).\)
In particular, \(\tau_{\mathbf M}(x)\in M_i\).
Since $\tau_{\mathbf M}(x)$ is a positive idempotent of $\mathbf M$, we have $\tau_{\mathbf M}(x)\in P_i$, so \(x\in \layer{P_i}\).
Hence $M_i\subseteq \layer{P_i}.$

Conversely, let \(x\in \layer{P_i}\).
Then \(\tau_{\mathbf M}(x)\in P_i\subseteq M_i\).
If \(x\in M_j\), then \eqref{eq:construction-local-units-preserved} gives \(\tau_{\mathbf M}(x)=\tau_{\mathbf M_j}(x)\in M_j.\)
Hence \(\tau_{\mathbf M}(x)\in M_i\cap M_j\).
Since the union is disjoint, \(i=j\).
Therefore \(x\in M_i\), so $\layer{P_i}\subseteq M_i.$
Thus $\layer{P_i}=M_i.$

For \eqref{item:level-partition-coherent}, let \(i,j\in I\), put \(k:=\max\{i,j\}\), and take any \(u\in \tprod{P_i}{P_j}\).
By definition there exist \(x\in \layer{P_i}=M_i\) and \(y\in \layer{P_j}=M_j\) such that \(u=\tau_{\mathbf M}(xy).\)
By the multiplication rule \eqref{defMULT}, \(xy=\rho_{i,k}(x)\cdot_k \rho_{j,k}(y)\in M_k.\)
Applying \eqref{eq:construction-local-units-preserved} to \(xy\in M_k\), we get \(u=\tau_{\mathbf M}(xy)=\tau_{\mathbf M_k}(xy)\in M_k.\)
Since \(u\) is a positive idempotent of \(\mathbf M\), it follows that \(u\in P_k\).
Therefore \(\tprod{P_i}{P_j}\subseteq P_k.\)

Since \(\tprod{P_i}{P_j}\neq\varnothing\), it follows that for every pair of blocks \(P_i,P_j\) there is a unique block \(P_k\) containing \(\tprod{P_i}{P_j}\), namely the one with \(k=\max\{i,j\}\).
Hence \(\mathcal P_{\mathrm{lev}}\) is \(\tau\)-multiplication-coherent.

Finally, \eqref{item:level-partition-canonical-refines} follows from Theorem~\ref{thm:finest-coherent}: since \(\mathcal P_{\mathrm{lev}}\) is \(\tau\)-multiplication-coherent, the finest such partition \(\Cm(\mathbf M)\) refines it.
Thus each block \(A\in\Cm(\mathbf M)\) is contained in a unique \(P_i\), and taking \(\tau^{-1}\) gives $\layer{A}\subseteq \layer{P_i}=M_i.$
This proves the final claim.
\end{proof}

The following lemma packages the restricted computations needed later: within one construction level, positive idempotents, $\tau$-layers, and $\tau$-saturated products can be computed either internally or in the reconstructed monoid.

\begin{lemma}[Construction-level calculus]\label{lem:construction-level-calculus}
In the setting of Proposition~\ref{prop:construction-level-partition-coherent}, for each \(i\in I\) and all \(U,V\subseteq P_i\), the following hold.
\begin{enumerate}
\item\label{item:construction-level-calculus-posid} \(P_i=\PosId(\mathbf M_i)\) and \(\layer{P_i}=M_i\).
\item\label{item:construction-level-calculus-layers} The intrinsic \(U\)-layer in \(\mathbf M_i\) is the same as the ambient \(U\)-layer in \(\mathbf M\).
\item\label{item:construction-level-calculus-tprod} The $\tau$-saturated product of \(U\) and \(V\) computed in \(\mathbf M_i\) coincides with the ambient one computed in \(\mathbf M\).
\end{enumerate}
\end{lemma}

\begin{proof}
Item~\ref{item:construction-level-calculus-posid} is Proposition~\ref{prop:construction-level-partition-coherent} \eqref{item:level-partition-internal-posid} and \eqref{item:level-partition-layers}.
For item~\ref{item:construction-level-calculus-layers}, let \(U\subseteq P_i\).
If \(x\) lies in the ambient \(U\)-layer, then \(\tau_{\mathbf M}(x)\in U\subseteq P_i\), so \(x\in\layer{P_i}=M_i\) by item~\ref{item:construction-level-calculus-posid}.
For \(x\in M_i\), Theorem~\ref{thm:construction-from-direct-system}\ref{item:construction-local-units} gives \(\tau_{\mathbf M}(x)=\tau_{\mathbf M_i}(x)\).
Hence
\[
\{x\in M_i:\tau_{\mathbf M_i}(x)\in U\}
=
\{x\in M:\tau_{\mathbf M}(x)\in U\}.
\]
Thus the intrinsic and ambient \(U\)-layers coincide.
The product assertion then follows because the \(U\)- and \(V\)-layers coincide and because multiplication and the local-unit map on \(M_i\) are the restrictions of the ambient multiplication and local-unit map.
\end{proof}

\section{A canonical rigid direct-system representation theorem}\label{sect:representation} 

The forward decomposition theorem gives a canonical direct system, whereas the converse construction requires strict compatibility in order to recover an ordered monoid.
To close this loop, one must verify that the canonical transition maps satisfy that condition automatically.
The next proposition is precisely this bridge; it both places the canonical system within the scope of the converse theorem and provides the input for the finite rigidity argument.

\begin{proposition}\label{prop:canonical-strict-compatibility}
Let $\mathbf M$ be a finite local-unit-aligned totally ordered monoid, and let
\[
\mathfrak D(\mathbf M)
=
\Bigl((\mathbf M_A)_{A\in\Cm(\mathbf M)},(\rho_{A,B})_{A\le_{\min} B}\Bigr)
\]
be its canonical direct-system decomposition.

Then \(\mathfrak D(\mathbf M)\) is strictly compatible in the sense of Definition~\ref{def:strict-compatibility}.

More explicitly, whenever \(B<_{\min} A\) and \(C<_{\min} A\), the following hold for all \(x\in M_A\), \(x'\in M_B\), and \(y\in M_C\):
\begin{enumerate}
\taggeditem{\textup{(R)}}{item:canonical-strict-compatibility-right}
\makebox[\linewidth][c]{%
\(\displaystyle
x<_{A}\rho_{B,A}(x')
\quad\Longrightarrow\quad
x\cdot_A \rho_{C,A}(y)
<_A
\rho_{B,A}(x')\cdot_A \rho_{C,A}(y);
\)}

\taggeditem{\textup{(L)}}{item:canonical-strict-compatibility-left}
\makebox[\linewidth][c]{%
\(\displaystyle
x<_{A}\rho_{B,A}(x')
\quad\Longrightarrow\quad
\rho_{C,A}(y)\cdot_A x
<_A
\rho_{C,A}(y)\cdot_A \rho_{B,A}(x').
\)}
\end{enumerate}
\end{proposition}

\begin{proof}
By Lemma~\ref{lem:component-calculus}\eqref{item:component-calculus-submonoid}, order and multiplication inside each canonical component agree with the corresponding ambient restricted operations.
We therefore use ambient isotonicity and then transport the result back to the common upper component.

Assume \(B<_{\min}A\), \(C<_{\min}A\), \(x\in M_A\), \(x'\in M_B\), \(y\in M_C\), and \(x<_{A}\rho_{B,A}(x')\).
By Theorem~\ref{thm:canonical-direct-system}\eqref{item:direct-order-recovery}, we have \(x\le x'\) in the ambient order; since \(A\ne B\), the corresponding components are disjoint, so in fact \(x<x'\).
Put \(D:=\max_{\le_{\min}}\{B,C\}\), so \(D<_{\min}A\).

Right multiplication by \(y\) gives \(xy\le x'y\).
By product recovery, \(xy=x\cdot_A\rho_{C,A}(y)\in M_A\), while \(x'y=\rho_{B,D}(x')\cdot_D\rho_{C,D}(y)\in M_D\).
Since \(D<_{\min}A\), order recovery applied to the comparison between the \(M_A\)-element \(xy\) and the \(M_D\)-element \(x'y\) gives
\[
x\cdot_A\rho_{C,A}(y)
<_A
\rho_{D,A}(x'y).
\]
By functoriality and multiplicativity of the transition maps,
\[
\rho_{D,A}(x'y)
=
\rho_{B,A}(x')\cdot_A\rho_{C,A}(y).
\]
Thus \ref{item:canonical-strict-compatibility-right} follows.

The left-handed argument is analogous.
From \(yx\le yx'\), product recovery gives \(yx=\rho_{C,A}(y)\cdot_A x\in M_A\) and \(yx'=\rho_{C,D}(y)\cdot_D\rho_{B,D}(x')\in M_D\).
Since \(D<_{\min}A\), order recovery gives
\[
\rho_{C,A}(y)\cdot_A x
<_A
\rho_{D,A}(yx'),
\]
and functoriality and multiplicativity give
\[
\rho_{D,A}(yx')
=
\rho_{C,A}(y)\cdot_A\rho_{B,A}(x').
\]
Hence \ref{item:canonical-strict-compatibility-left} follows.
\end{proof}

\begin{corollary}\label{cor:canonical-decomposition-compatible} 
Every canonical direct-system decomposition arising from a finite local-unit-aligned totally ordered monoid satisfies the hypothesis of Theorem~\ref{thm:construction-from-direct-system}.
\end{corollary}

\begin{proof}
Let $\mathfrak D(\mathbf M) = \Bigl((\mathbf M_A)_{A\in\Cm(\mathbf M)},(\rho_{A,B})_{A\le_{\min} B}\Bigr)$ be the canonical direct-system decomposition of a finite local-unit-aligned totally ordered monoid \(\mathbf M\).

By Theorem~\ref{thm:canonical-direct-system}\eqref{item:direct-id}--\eqref{item:direct-monoid-hom}, \(\mathfrak D(\mathbf M)\) is a direct system of finite totally ordered monoids over the finite chain \(\bigl(\Cm(\mathbf M),\le_{\min}\bigr)\).
By Corollary~\ref{cor:all-components-local-unit-aligned-monoids}, each constituent \(\mathbf M_A\) is in fact a finite local-unit-aligned totally ordered monoid.
Finally, by Proposition~\ref{prop:canonical-strict-compatibility}, \(\mathfrak D(\mathbf M)\) is strictly compatible.

Thus \(\mathfrak D(\mathbf M)\) satisfies all hypotheses of Theorem~\ref{thm:construction-from-direct-system}.
\end{proof}

\begin{remark}\label{rem:strict-compatibility-role}
The constructive direction requires an additional strictness condition because ordinary isotonicity of the constituent monoids and transition maps does not by itself guarantee that the multiplication defined by \eqref{defMULT} is isotone with respect to the directed lexicographic order.
The exact obstruction is the reversed cross-level situation: one compares an element in a higher level with the image of an element from a lower level, and then multiplies by an element coming from a still lower level.
In that case, a weak inequality in the higher component is not sufficient; one needs the strict inequalities encoded in strict compatibility.

Lemma~\ref{lem:exact-finite-chain-criterion} shows that this is precisely the right condition: for finite chain-indexed direct systems, the multiplication defined by \eqref{defMULT} is isotone if and only if the system is strictly compatible.

On the other hand, Proposition~\ref{prop:canonical-strict-compatibility} shows that the canonical direct-system decomposition attached to an ambient finite local-unit-aligned totally ordered monoid satisfies strict compatibility automatically.
Thus the compatibility condition used in the constructive direction is not ad hoc: it is exactly the strict cross-level behavior inherited from the intrinsic canonical decomposition.
\end{remark}

\begin{definition}[Unit-constant homomorphism]\label{def:unit-constant-hom}
Let \( \varphi\colon \mathbf A\to \mathbf B \) be an isotone unital monoid homomorphism between finite totally ordered monoids.
We say that \(\varphi\) is \emph{unit-constant} if \(\varphi(x)=e_B\) for every \(x\in A\), where \(e_B\) denotes the unit of \(\mathbf B\).
\end{definition}

\begin{definition}[Unit-constant direct system]\label{def:unit-constant}
A direct system of finite local-unit-aligned totally ordered monoids over a finite chain is called \emph{unit-constant} if each of its proper transition maps is unit-constant.
\end{definition}

This condition isolates the degenerate upward-transport case in which every proper transition forgets the source element completely and retains only the identity of the target component.

\begin{corollary}\label{cor:unit-constant-strictly-compatible}
Every unit-constant direct system is strictly compatible.
\end{corollary}

\begin{proof}
Let \(i<j\) and \(k<j\), and let \(u\in M_j\), \(x\in M_i\), and \(y\in M_k\).
Since the proper transition maps are unit-constant, \(\rho_{i,j}(x)=e_j\) and \(\rho_{k,j}(y)=e_j\).
Thus the antecedent in Definition~\ref{def:strict-compatibility} is simply \(u<_j e_j\).
For \ref{item:strict-compatibility-right} we get
\[
u\cdot_j \rho_{k,j}(y)=u\cdot_j e_j=u
<_j e_j=e_j\cdot_j e_j
=\rho_{i,j}(x)\cdot_j \rho_{k,j}(y),
\]
and for \ref{item:strict-compatibility-left} we similarly get
\[
\rho_{k,j}(y)\cdot_j u=e_j\cdot_j u=u
<_j e_j=e_j\cdot_j e_j
=\rho_{k,j}(y)\cdot_j \rho_{i,j}(x).
\]
Hence both strict-compatibility conditions hold.
\end{proof}
The next proposition shows that, in the finite totally ordered setting, the converse implication also holds.

\begin{proposition}[Strict compatibility forces unit-constancy in the finite case]\label{prop:strict-implies-unit-constant}
Let $(I,\le)$ be a finite chain, and let \(\bigl((\mathbf M_i)_{i\in I},(\rho_{i,j})_{i\le j}\bigr)\) be a strictly compatible direct system of finite totally ordered monoids, where \(\mathbf M_i=\langle M_i,\le_i,\cdot_i,e_i\rangle.\)
Then for every $i<j$, \(\rho_{i,j}(x)=e_j\) for every \(x\in M_i\).
In other words, every proper transition map in the system is unit-constant.
\end{proposition}

\begin{proof}
Fix indices $i<j$.
We show that the map \(\rho_{i,j}\colon M_i\to M_j\) is constant with value $e_j$.

Assume first that there exists $x\in M_i$ with \(\rho_{i,j}(x)>e_j.\)
Since $\rho_{i,j}$ is isotone and unital, $\rho_{i,j}(e_i)=e_j$, so necessarily $x>e_i$.
Choose such an $x$ maximal; this is possible because $M_i$ is finite.

Because \(e_j<\rho_{i,j}(x),\) strict compatibility \ref{item:strict-compatibility-right} applies with \(u:=e_j\), \(x':=x\), \(y:=x\), and \(k:=i\).
Thus \(e_j\cdot_j \rho_{i,j}(x) < \rho_{i,j}(x)\cdot_j \rho_{i,j}(x).\)
Since $e_j$ is the unit of $\mathbf M_j$ and $\rho_{i,j}$ is a homomorphism, this becomes \(\rho_{i,j}(x)<\rho_{i,j}(x)^2=\rho_{i,j}(x^2).\)
Hence \(\rho_{i,j}(x^2)>e_j.\)

Now $x>e_i$, so by isotonicity of multiplication in $\mathbf M_i$, \(x=e_ix\le_i xx=x^2.\)
If $x^2=x$, then the strict inequality $\rho_{i,j}(x)<\rho_{i,j}(x^2)$ is impossible.
If $x^2>x$, then $x^2$ is a larger counterexample than $x$, contradicting the maximality of $x$.
Thus no such $x$ can exist.

Assume next that there exists $x\in M_i$ with \(\rho_{i,j}(x)<e_j.\)
Again, since $\rho_{i,j}(e_i)=e_j$ and $\rho_{i,j}$ is isotone, this implies $x<e_i$.
Choose such an $x$ minimal; this is possible because $M_i$ is finite.

Now \(\rho_{i,j}(x)<e_j=\rho_{i,j}(e_i),\) so strict compatibility \ref{item:strict-compatibility-right} applies with \(u:=\rho_{i,j}(x)\), \(x':=e_i\), \(y:=x\), and \(k:=i\).
We obtain \(\rho_{i,j}(x)\cdot_j \rho_{i,j}(x) < \rho_{i,j}(e_i)\cdot_j \rho_{i,j}(x).\)
Hence \(\rho_{i,j}(x^2)=\rho_{i,j}(x)^2<e_j\rho_{i,j}(x)=\rho_{i,j}(x),\) so in particular \(\rho_{i,j}(x^2)<e_j.\)

Since $x<e_i$, isotonicity in $\mathbf M_i$ gives \(x^2=xx\le_i xe_i=x.\)
If $x^2=x$, this contradicts the strict inequality $\rho_{i,j}(x^2)<\rho_{i,j}(x)$.
If $x^2<x$, then $x^2$ is a smaller counterexample than $x$, contradicting the minimality of $x$.
Thus this case is also impossible.

Therefore neither $\rho_{i,j}(x)>e_j$ nor $\rho_{i,j}(x)<e_j$ can occur for any $x\in M_i$.
Hence \(\rho_{i,j}(x)=e_j\) for every \(x\in M_i\), as claimed.
\end{proof}

Thus, in the finite chain-indexed setting of the present paper, strict compatibility is forced to coincide with unit-constancy of all proper transition maps.
We nevertheless keep the notions separate.
Strict compatibility is the order-theoretic condition naturally arising from isotonicity of the directed-lexicographic reconstruction, whereas unit-constancy is the rigid form that this condition takes in the finite setting considered here.
This distinction is useful conceptually and becomes relevant outside the finite setting, where strict compatibility and unit-constancy should no longer be conflated.

\begin{corollary}[Rigidity of canonical transition maps]
\label{cor:canonical-transition-maps-unit-constant}
Let
\[
\mathfrak D(\mathbf M)
=
\Bigl((\mathbf M_{A_i})_{i=0}^{n},(\rho_{A_i,A_j})_{0\le i\le j\le n}\Bigr)
\]
be the canonical direct-system decomposition of a finite local-unit-aligned totally ordered monoid \(\mathbf M\), where \(A_0<_{\min}A_1<_{\min}\cdots<_{\min}A_n.\)
Then for every \(i<j\), the canonical transition map \(\rho_{A_i,A_j}\colon \mathbf M_{A_i}\to \mathbf M_{A_j}\) is unit-constant: \(\rho_{A_i,A_j}(x)=e_{A_j}\) for every \(x\in M_{A_i}\).
Equivalently, the canonical direct-system decomposition of \(\mathbf M\) is a unit-constant direct system.
\end{corollary}

\begin{proof}
By Proposition~\ref{prop:canonical-strict-compatibility}, the canonical direct system of \(\mathbf M\) is strictly compatible.
Therefore the claim is a special case of Proposition~\ref{prop:strict-implies-unit-constant}.
\end{proof}

\begin{corollary}[Rigid cross-level behaviour]
\label{cor:rigid-cross-level-behaviour-direct}
Let \(A<_{\min}B\), \(x\in M_A\), and \(y\in M_B\).
Then \(xy=y=yx,\) and \(x\le y \iff e_B\le y.\)
\end{corollary}

\begin{proof}
By Corollary~\ref{cor:canonical-transition-maps-unit-constant}, \(\rho_{A,B}(x)=e_B.\)
Now Theorem~\ref{thm:canonical-direct-system}\eqref{item:direct-product-recovery} gives \(xy=\rho_{A,B}(x)\rho_{B,B}(y)=e_By=y,\) and similarly $yx=y$.
Also Theorem~\ref{thm:canonical-direct-system}\eqref{item:direct-upward-comparison} yields \(x\le y \iff \rho_{A,B}(x)\le y \iff e_B\le y.\)
\end{proof}

\begin{corollary}[The canonical components are $\tau$-multiplication-cohesive]
\label{cor:canonical-components-already-terminal-direct}
For every $A\in\Cm(\mathbf M)$, the component monoid $\mathbf M_A$ is $\tau$-multiplication-cohesive.
\end{corollary}

\begin{proof}
Assume, towards a contradiction, that for some \(A\in\Cm(\mathbf M)\), the component monoid \(\mathbf M_A\) is not \(\tau\)-multiplication-cohesive.
Then there exists a nontrivial \(\tau\)-multiplication-coherent partition \(\mathcal Q\) of \(\PosId(\mathbf M_A)=A\) inside \(\mathbf M_A\).
Define \(\mathcal P':=(\Cm(\mathbf M)\setminus\{A\})\cup\mathcal Q\).
Since \(\mathcal Q\) is a partition of \(A\), and the remaining blocks of \(\Cm(\mathbf M)\) partition \(\PosId(\mathbf M)\setminus A\), the family \(\mathcal P'\) is a partition of \(\PosId(\mathbf M)\).
Thus \(\mathcal P'\) is obtained from the canonical partition \(\Cm(\mathbf M)\) by splitting the single block \(A\) according to the chosen intrinsic \(\tau\)-multiplication-coherent partition \(\mathcal Q\) of \(\mathbf M_A\).

We claim that $\mathcal P'$ is $\tau$-multiplication-coherent.

First, if $U,V\in\mathcal Q$, then by Lemma~\ref{lem:intrinsic-tau-on-component}\eqref{item:component-layers} and \eqref{item:component-tprod}, the ambient and intrinsic $\tau$-layers and $\tau$-saturated products coincide on $A$.
Since $\mathcal Q$ is $\tau$-multiplication-coherent in $\mathbf M_A$, \((\tprod{U}{V})^{\mathbf M} = (\tprod{U}{V})^{\mathbf M_A}\) is contained in a unique block of $\mathcal Q$.

Next, let $U\in\mathcal Q$ and let $C\in\Cm(\mathbf M)$ with $C\neq A$.

If $C<_{\min}A$, then for \(x\in M_C=\layer{C}\) and \(y\in\layer{U}\subseteq M_A\), Corollary~\ref{cor:rigid-cross-level-behaviour-direct} gives \(xy=y=yx\).
Hence \(\tau(xy)=\tau(y)\in U\) and \(\tau(yx)=\tau(y)\in U\), so $ (\tprod{C}{U})^{\mathbf M}\subseteq U $ and $ (\tprod{U}{C})^{\mathbf M}\subseteq U $.
Conversely, let \(u\in U\).
Choose \(c\in C\).
Since \(c\) and \(u\) are positive idempotents, \(c\in\layer{C}\) and \(u\in\layer{U}\).
Applying the same cross-level identity to \(c\) and \(u\), we get \(cu=u=uc\), whence \(\tau(cu)=u=\tau(uc)\).
Thus $ U\subseteq(\tprod{C}{U})^{\mathbf M} $ and $ U\subseteq(\tprod{U}{C})^{\mathbf M} $.
Therefore \((\tprod{C}{U})^{\mathbf M}=U=(\tprod{U}{C})^{\mathbf M}.\)

If \(A<_{\min}C\), then for \(x\in\layer{U}\subseteq M_A\) and \(y\in M_C=\layer{C}\), Corollary~\ref{cor:rigid-cross-level-behaviour-direct} gives \(xy=y=yx\).
Hence \(\tau(xy)=\tau(y)\in C\) and \(\tau(yx)=\tau(y)\in C\), so $ (\tprod{U}{C})^{\mathbf M}\subseteq C $ and $ (\tprod{C}{U})^{\mathbf M}\subseteq C $.
Conversely, let \(c\in C\).
Choose \(u\in U\).
Since \(u\) and \(c\) are positive idempotents, \(u\in\layer{U}\) and \(c\in\layer{C}\).
Applying the same cross-level identity to \(u\) and \(c\), we get \(uc=c=cu\), whence \(\tau(uc)=c=\tau(cu)\).
Thus $ C\subseteq(\tprod{U}{C})^{\mathbf M} $ and $ C\subseteq(\tprod{C}{U})^{\mathbf M} $.
Therefore \((\tprod{U}{C})^{\mathbf M}=C=(\tprod{C}{U})^{\mathbf M}.\)

Finally, if $C,D\in\Cm(\mathbf M)$ are both distinct from $A$, then $\Cm(\mathbf M)$ itself is $\tau$-multiplication-coherent, so \((\tprod{C}{D})^{\mathbf M}\) is contained in the unique block \(C\vee D=\max\nolimits_{\le_{\min}}\{C,D\}\) of $\Cm(\mathbf M)$.
Since neither $C$ nor $D$ is $A$, this unique target block is also different from $A$, hence remains a block of $\mathcal P'$.

Thus $\mathcal P'$ is $\tau$-multiplication-coherent.

Now $\Cm(\mathbf M)$ is the finest $\tau$-multiplication-coherent partition by Theorem~\ref{thm:finest-coherent}, so $\Cm(\mathbf M)$ refines $\mathcal P'$.
But by construction $\mathcal P'$ refines $\Cm(\mathbf M)$.
Therefore \(\mathcal P'=\Cm(\mathbf M).\)
This is impossible because $\mathcal Q$ was assumed nontrivial.
Hence no such $A$ exists.

Therefore every $\mathbf M_A$ is $\tau$-multiplication-cohesive.
\end{proof}

We now arrive at the main representation theorem.
The point of the preceding rigidity results is that the direct-system representation obtained from the canonical $\tau$-multiplication-coherent decomposition is not merely a chain-indexed representation with compatible transition maps.
Lemma~\ref{lem:exact-finite-chain-criterion} identifies strict compatibility as exactly the condition needed for the directed-lexicographic reconstruction to be an ordered monoid, while Proposition~\ref{prop:strict-implies-unit-constant} shows that, in the finite totally ordered setting, this condition forces every proper transition map to be unit-constant.
Consequently the cross-level multiplication is absorptive, as in Corollary~\ref{cor:rigid-cross-level-behaviour-direct}, and the canonical components are already intrinsically $\tau$-multiplication-cohesive, as in Corollary~\ref{cor:canonical-components-already-terminal-direct}.
Thus the final theorem is not just a direct-system representation theorem, but a canonical rigid direct-system representation theorem.

\begin{theorem}[Canonical rigid direct-system representation theorem]
\label{thm:representation}
The following hold.

\begin{enumerate}
\taggeditem{\textup{(i)}}{item:representation-canonical-direct-system}
Every finite local-unit-aligned totally ordered monoid \( \mathbf M=\langle M,\le,\cdot,e\rangle \) admits a canonical unit-constant direct system
\[
\mathfrak D(\mathbf M)
=
\Bigl((\mathbf M_A)_{A\in\Cm(\mathbf M)},(\rho_{A,B})_{A\le_{\min} B}\Bigr)
\]
of \(\tau\)-multiplication-cohesive finite local-unit-aligned totally ordered monoids over the finite chain \( (\Cm(\mathbf M),\le_{\min}). \)

\taggeditem{\textup{(ii)}}{item:representation-reconstruction-from-direct-system}
Conversely, every unit-constant direct system \( \mathbf D=\bigl((\mathbf M_i)_{i\in I},(\rho_{i,j})_{i\le j}\bigr) \) of finite local-unit-aligned totally ordered monoids over a finite chain \((I,\le)\), with no cohesiveness assumption on the constituents, canonically determines a finite local-unit-aligned totally ordered monoid \(\mathfrak M(\mathbf D)=\langle M,\le,\cdot,e\rangle\) whose universe is the disjoint union of the \(M_i\)'s, whose order is the directed lexicographic order, and whose multiplication is given by
\[
x\in M_i,\ y\in M_j,\ k:=\max\{i,j\},
\qquad
xy:=\rho_{i,k}(x)\cdot_k \rho_{j,k}(y).
\]

\taggeditem{\textup{(iii)}}{item:representation-canonical-recovery}
If \(\mathbf M\) is a finite local-unit-aligned totally ordered monoid, then \(\mathfrak M(\mathfrak D(\mathbf M))\cong \mathbf M\) canonically.

\taggeditem{\textup{(iv)}}{item:representation-direct-system-recovery}
If \(\mathbf D=\bigl((\mathbf M_i)_{i\in I},(\rho_{i,j})_{i\le j}\bigr)\) is a unit-constant direct system of finite local-unit-aligned totally ordered monoids, then the canonical decomposition \(\Lm(\mathfrak M(\mathbf D))\) of \(\mathfrak M(\mathbf D)\) refines the given construction-level partition \(\{M_i:i\in I\}\).

Moreover, if every constituent monoid \(\mathbf M_i\) is \(\tau\)-multiplication-cohesive, then the construction-level layer decomposition is exactly the layer decomposition induced by the canonical \(\tau\)-multiplication-coherent partition of \(\mathfrak M(\mathbf D)\).
More precisely, putting \(P_i:=\PosId(\mathfrak M(\mathbf D))\cap M_i,\) we have
\[
\Cm(\mathfrak M(\mathbf D))=\{P_i:i\in I\},
\qquad
\Lm(\mathfrak M(\mathbf D))=\{M_i:i\in I\}.
\]
Thus, under this cohesiveness hypothesis, the construction-level partition is not merely refined by the canonical decomposition; it is canonical.

In this case the canonical direct system of \(\mathfrak M(\mathbf D)\) recovers the original direct system: \(\mathfrak D(\mathfrak M(\mathbf D))\cong \mathbf D\) canonically.
More explicitly, the map \(i\mapsto P_i\) is an order isomorphism from \(I\) onto \(\Cm(\mathfrak M(\mathbf D))\), the constituent monoid indexed by \(P_i\) is exactly \(\mathbf M_i\), and for \(i\le j\) the canonical transition map \(\widehat\rho_{P_i,P_j}\colon M_i\to M_j\) coincides with the original transition map \(\rho_{i,j}\).

\end{enumerate}
\end{theorem}

\begin{proof}
For item~\ref{item:representation-canonical-direct-system}, Theorem~\ref{thm:canonical-direct-system}\eqref{item:direct-id}--\eqref{item:direct-monoid-hom} yields the canonical direct system
\[
\mathfrak D(\mathbf M)
=
\Bigl((\mathbf M_A)_{A\in\Cm(\mathbf M)},(\rho_{A,B})_{A\le_{\min} B}\Bigr)
\]
over the finite chain \((\Cm(\mathbf M),\le_{\min})\).
By Corollary~\ref{cor:canonical-transition-maps-unit-constant}, every proper canonical transition map is unit-constant, so \(\mathfrak D(\mathbf M)\) is a unit-constant direct system.
By Corollary~\ref{cor:canonical-components-already-terminal-direct}, every constituent monoid \(\mathbf M_A\) is \(\tau\)-multiplication-cohesive.

For item~\ref{item:representation-reconstruction-from-direct-system}, let \(\mathbf D=\bigl((\mathbf M_i)_{i\in I},(\rho_{i,j})_{i\le j}\bigr)\) be a unit-constant direct system.
By Corollary~\ref{cor:unit-constant-strictly-compatible}, \(\mathbf D\) is strictly compatible.
Hence Theorem~\ref{thm:construction-from-direct-system} applies and, by \ref{item:construction-order-condition} and \ref{item:construction-multiplication}, yields a finite local-unit-aligned totally ordered monoid \(\mathfrak M(\mathbf D)=\langle M,\le,\cdot,e\rangle\) with universe \(M=\bigsqcup_{i\in I} M_i\), directed lexicographic order, and multiplication
\[
xy:=\rho_{i,k}(x)\cdot_k \rho_{j,k}(y)
\qquad
(x\in M_i,\ y\in M_j,\ k=\max\{i,j\}).
\]

For item~\ref{item:representation-canonical-recovery}, apply item~\ref{item:representation-reconstruction-from-direct-system} to the canonical direct system \(\mathfrak D(\mathbf M)\) from item~\ref{item:representation-canonical-direct-system}, and write
\[
\mathfrak M(\mathfrak D(\mathbf M))
=
\langle M^\ast,\le^\ast,\cdot^\ast,e^\ast\rangle .
\]
Thus \(M^\ast=\bigsqcup_{A\in\Cm(\mathbf M)} M_A = \bigsqcup_{A\in\Cm(\mathbf M)} \layer{A},\) where the union is understood as an external tagged disjoint union.
We write an element of the copy of \(\layer{A}\) as \((A,x)\), with \(x\in\layer{A}\).

Define \(\Phi\colon M^\ast\to M\) by \(\Phi(A,x):=x\).
This map is well defined and bijective, because \(\{\layer{A}:A\in\Cm(\mathbf M)\}\) is a partition of \(M\) by Theorem~\ref{thm:canonical-direct-system}\eqref{item:direct-disjoint-union}.

We first verify that \(\Phi\) preserves and reflects the order.
Let \(x\in M_A=\layer{A}\) and \(y\in M_B=\layer{B}\).
If \(A=B\), then
\[
(A,x)\le^\ast(A,y)
\quad\Longleftrightarrow\quad
x\le y \text{ inside }\mathbf M_A
\quad\Longleftrightarrow\quad
x\le y \text{ in }\mathbf M.
\]
If \(A<_{\min}B\), then the directed lexicographic order in \(\mathfrak M(\mathfrak D(\mathbf M))\) gives
\[
(A,x)\le^\ast(B,y)
\quad\Longleftrightarrow\quad
\rho_{A,B}(x)\le y \text{ inside }\mathbf M_B.
\]
By Theorem~\ref{thm:canonical-direct-system}\eqref{item:direct-order-recovery}, this is equivalent to \(x\le y\) in \(\mathbf M\).
Finally, if \(B<_{\min}A\), then the directed lexicographic order gives
\[
(A,x)\le^\ast(B,y)
\quad\Longleftrightarrow\quad
x<\rho_{B,A}(y) \text{ inside }\mathbf M_A,
\]
which is again equivalent to \(x\le y\) in \(\mathbf M\) by Theorem~\ref{thm:canonical-direct-system}\eqref{item:direct-order-recovery}.
Thus \(\Phi\) is an order isomorphism.

Next we verify multiplication.
Let \(x\in M_A\), \(y\in M_B\), and put \(C:=A\vee B=\max_{\le_{\min}}\{A,B\}\).
By the multiplication rule in Theorem~\ref{thm:construction-from-direct-system}\ref{item:construction-multiplication}, \((A,x)\cdot^\ast(B,y) = \bigl(C,\rho_{A,C}(x)\rho_{B,C}(y)\bigr),\) where the product on the right is computed inside \(\mathbf M_C\).
Therefore
\[
\Phi\bigl((A,x)\cdot^\ast(B,y)\bigr)
=
\rho_{A,C}(x)\rho_{B,C}(y)
=
xy
=
\Phi(A,x)\Phi(B,y),
\]
using Theorem~\ref{thm:canonical-direct-system}\eqref{item:direct-product-recovery}.
Hence \(\Phi\) is multiplicative.

Finally, let \(A_0:=\min\Cm(\mathbf M)\), and let \(E\in\Cm(\mathbf M)\) be the canonical block containing \(e\).
Since positive idempotents are ordered above the identity, \(e\) is the least positive idempotent of \(\mathbf M\).
Hence \(\min E=e\), and for every \(A\in\Cm(\mathbf M)\) we have \(e\le \min A\).
Thus \(E\) is the unique canonical block with least minimum, so \(E=A_0\).
Consequently \(e\in A_0\), and the block unit of \(A_0\) is \(e_{A_0}=\min A_0=e\).
The identity of \(\mathfrak M(\mathfrak D(\mathbf M))\) is therefore \((A_0,e)\), and \(\Phi(A_0,e)=e.\)
Thus \(\Phi\) is a unital ordered-monoid isomorphism \(\mathfrak M(\mathfrak D(\mathbf M))\cong \mathbf M.\)
The construction of \(\Phi\) uses only the canonical decomposition, so the isomorphism is canonical.

For item~\ref{item:representation-direct-system-recovery}, let \(\mathbf D=\bigl((\mathbf M_i)_{i\in I},(\rho_{i,j})_{i\le j}\bigr)\) be a unit-constant direct system, and let \(\mathfrak M(\mathbf D)=\langle M,\le,\cdot,e\rangle\) be the monoid obtained in item~\ref{item:representation-reconstruction-from-direct-system}.
Since unit-constant systems are strictly compatible, Proposition~\ref{prop:construction-level-partition-coherent}\eqref{item:level-partition-canonical-refines} applies.
Hence \(\Cm(\mathfrak M(\mathbf D))\) refines the construction-level partition \(\{P_i:i\in I\}\), where \(P_i:=\PosId(\mathfrak M(\mathbf D))\cap M_i\), and therefore \(\Lm(\mathfrak M(\mathbf D))\) refines \(\{M_i:i\in I\}\).

Assume now that every \(\mathbf M_i\) is $\tau$-multiplication-cohesive.
For \(i\in I\), let \(\mathcal Q_i:=\{A\in\Cm(\mathfrak M(\mathbf D)):A\subseteq P_i\}\).
This is a partition of \(P_i\).
We claim that \(\mathcal Q_i\) is $\tau$-multiplication-coherent inside \(\mathbf M_i\).
Indeed, by Lemma~\ref{lem:construction-level-calculus}\eqref{item:construction-level-calculus-posid}--\eqref{item:construction-level-calculus-tprod}, \(P_i=\PosId(\mathbf M_i)\), and for all \(A,B\in\mathcal Q_i\) the \(A\)- and \(B\)-layers, as well as the corresponding $\tau$-saturated product, are the same whether computed inside \(\mathbf M_i\) or ambiently in \(\mathfrak M(\mathbf D)\).
Hence
\[
(\tprod{A}{B})^{\mathbf M_i}
=
(\tprod{A}{B})^{\mathfrak M(\mathbf D)}.
\]
This set is nonempty and, since \(\Cm(\mathfrak M(\mathbf D))\) is ambiently $\tau$-multiplication-coherent, it is contained in a unique canonical block \(C\in\Cm(\mathfrak M(\mathbf D))\).
On the other hand, because the product is computed inside \(\mathbf M_i\), all its values lie in \(\PosId(\mathbf M_i)=P_i\).
Thus \(C\cap P_i\ne\varnothing\).
By Proposition~\ref{prop:construction-level-partition-coherent} \eqref{item:level-partition-canonical-refines}, every canonical block is contained in a unique construction-level positive-idempotent block.
Since \(C\) meets \(P_i\), that block must be \(P_i\).
Hence \(C\subseteq P_i\), so \(C\in\mathcal Q_i\).
This proves that \(\mathcal Q_i\) is $\tau$-multiplication-coherent in \(\mathbf M_i\).

Cohesiveness of \(\mathbf M_i\) now forces \(\mathcal Q_i=\{P_i\}\) for every \(i\).
Consequently \(\Cm(\mathfrak M(\mathbf D))=\{P_i:i\in I\}\), and Lemma~\ref{lem:construction-level-calculus}\eqref{item:construction-level-calculus-posid} gives \(\Lm(\mathfrak M(\mathbf D))=\{M_i:i\in I\}\).

It remains only to identify the recovered direct system.
The map \(i\mapsto P_i\) is an order isomorphism onto \(\Cm(\mathfrak M(\mathbf D))\).
Indeed, by Proposition~\ref{prop:construction-level-partition-coherent}\eqref{item:level-partition-internal-posid}, \(P_i=\PosId(\mathbf M_i)\), and the least element of \(P_i\) is the identity \(e_i\).
If \(i<j\), then \(\rho_{i,j}(e_i)=e_j\).
Hence, in the directed lexicographic order, \(e_i<e_j\): the transported values agree and the index \(i\) precedes \(j\).
Conversely, if \(e_i<e_j\), then \(i<j\), since the case \(j<i\) would give \(e_j<e_i\) by the same argument, contradicting antisymmetry.
Thus the ordering of the canonical blocks by their minima agrees exactly with the ordering of \(I\).
The constituent monoid indexed by \(P_i\) is exactly \(\mathbf M_i\), because \(\layer{P_i}=M_i\) by Lemma~\ref{lem:construction-level-calculus}\eqref{item:construction-level-calculus-posid}, the order and multiplication restrict to \(\le_i\) and \(\cdot_i\), and the block unit is \(e_i\).

Finally, if \(i\le j\), the canonical transition map \(\widehat\rho_{P_i,P_j}\colon M_i\to M_j\) is given by \(\widehat\rho_{P_i,P_j}(x)=x e_{P_j}=x e_j\).
The multiplication rule in \(\mathfrak M(\mathbf D)\) gives \(x e_j=\rho_{i,j}(x)\cdot_j e_j=\rho_{i,j}(x)\).
Thus \(\widehat\rho_{P_i,P_j}=\rho_{i,j}\), and under the order isomorphism \(i\mapsto P_i\) the recovered canonical direct system is canonically isomorphic to \(\mathbf D\).
\end{proof}

Theorem~\ref{thm:representation} is where the three main strands of the paper meet: intrinsic decomposition, converse reconstruction, and finite rigidity.
It shows that the direct-system structure attached to a finite local-unit-aligned totally ordered monoid is not only descriptive, but reconstructs the monoid canonically within the present theory.

\begin{remark}[Comparison with Clifford's theorem for inverse semigroups]
The classical Clifford theorem characterizes Clifford inverse semigroups, equivalently, inverse semigroups with central idempotents, as strong semilattices of groups \cite{Clifford1941,CliffordPreston1961}.
The idempotent semilattice indexes the maximal subgroups, and coherent connecting homomorphisms determine mixed products.
Theorem~\ref{thm:representation}\ref{item:representation-canonical-direct-system} and \ref{item:representation-canonical-recovery} have the same structural pattern: a global monoid is recovered canonically from components and connecting maps.

The differences locate both the scope and the additional content of the present result.
Our components need not be groups and carry no inverse operation; the index is a finite chain of canonical $\tau$-multiplication-coherent blocks rather than an the whole idempotent semilattice; and the directed lexicographic construction recovers the order as well as the multiplication.
Most importantly, unit-constancy is not part of Clifford's theorem: the connecting homomorphisms in a strong semilattice of groups may be nontrivial.
Here Proposition~\ref{prop:strict-implies-unit-constant} shows that finiteness and total order force every strictly compatible proper transition map to be unit-constant.
This rigidity is therefore an additional conclusion beyond the Clifford analogy, and it turns the strong-semilattice-like reconstruction into the ordinal-sum-like absorption law \(xy=y=yx\) across distinct levels, as recorded in Corollary~\ref{cor:rigid-cross-level-behaviour-direct}.
\end{remark}

\definecolor{compA}{RGB}{0,0,0}
\definecolor{compB}{RGB}{0,90,160}
\definecolor{compC}{RGB}{180,90,0}
\definecolor{compD}{RGB}{0,120,70}
\definecolor{compE}{RGB}{150,0,90}
\definecolor{compF}{RGB}{120,70,0}
\definecolor{compG}{RGB}{90,0,140}

\newcommand{\A}[1]{\textcolor{compA}{#1}}
\newcommand{\B}[1]{\textcolor{compB}{#1}}
\newcommand{\Cszin}[1]{\textcolor{compC}{#1}}
\newcommand{\D}[1]{\textcolor{compD}{#1}}
\newcommand{\E}[1]{\textcolor{compE}{#1}}
\newcommand{\F}[1]{\textcolor{compF}{#1}}
\newcommand{\G}[1]{\textcolor{compG}{#1}}

\begin{figure}[ht]
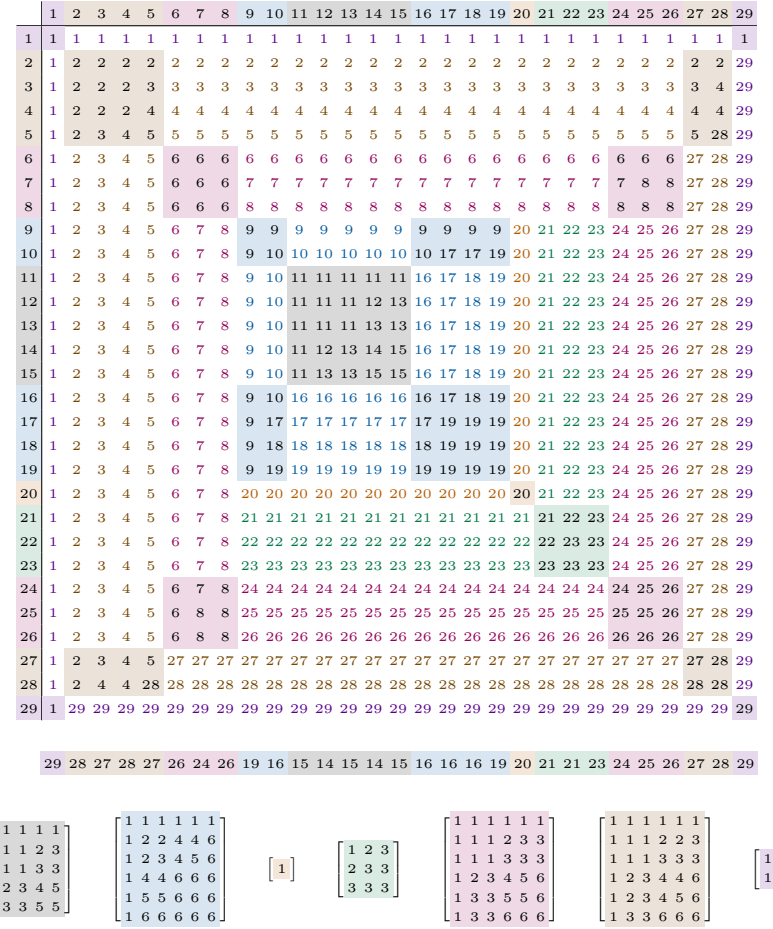

\centering
{\tiny

{
\setlength{\arraycolsep}{1.25pt}
\renewcommand{\arraystretch}{1.50}

\[
\begin{array}{c}
\begin{array}{c|*{29}{c}}
&{\cellcolor{compG!15}1}&{\cellcolor{compF!15}2}&{\cellcolor{compF!15}3}&{\cellcolor{compF!15}4}&{\cellcolor{compF!15}5}&{\cellcolor{compE!15}6}&{\cellcolor{compE!15}7}&{\cellcolor{compE!15}8}&{\cellcolor{compB!15}9}&{\cellcolor{compB!15}10}&{\cellcolor{compA!15}11}&{\cellcolor{compA!15}12}&{\cellcolor{compA!15}13}&{\cellcolor{compA!15}14}&{\cellcolor{compA!15}15}&{\cellcolor{compB!15}16}&{\cellcolor{compB!15}17}&{\cellcolor{compB!15}18}&{\cellcolor{compB!15}19}&{\cellcolor{compC!15}20}&{\cellcolor{compD!15}21}&{\cellcolor{compD!15}22}&{\cellcolor{compD!15}23}&{\cellcolor{compE!15}24}&{\cellcolor{compE!15}25}&{\cellcolor{compE!15}26}&{\cellcolor{compF!15}27}&{\cellcolor{compF!15}28}&{\cellcolor{compG!15}29}\\
\hline
{\cellcolor{compG!15}1}&{\cellcolor{compG!15}\,1\, }&\G{1}&\G{1}&\G{1}&\G{1}&\G{1}&\G{1}&\G{1}&\G{1}&\G{1}&\G{1}&\G{1}&\G{1}&\G{1}&\G{1}&\G{1}&\G{1}&\G{1}&\G{1}&\G{1}&\G{1}&\G{1}&\G{1}&\G{1}&\G{1}&\G{1}&\G{1}&\G{1}&{\cellcolor{compG!15}1}\\
{\cellcolor{compF!15}2}&\G{1}&{\cellcolor{compF!15}2}&{\cellcolor{compF!15}2}&{\cellcolor{compF!15}2}&{\cellcolor{compF!15}2}&\F{2}&\F{2}&\F{2}&\F{2}&\F{2}&\F{2}&\F{2}&\F{2}&\F{2}&\F{2}&\F{2}&\F{2}&\F{2}&\F{2}&\F{2}&\F{2}&\F{2}&\F{2}&\F{2}&\F{2}&\F{2}&{\cellcolor{compF!15}2}&{\cellcolor{compF!15}2}&\G{29}\\
{\cellcolor{compF!15}3}&\G{1}&{\cellcolor{compF!15}2}&{\cellcolor{compF!15}2}&{\cellcolor{compF!15}2}&{\cellcolor{compF!15}3}&\F{3}&\F{3}&\F{3}&\F{3}&\F{3}&\F{3}&\F{3}&\F{3}&\F{3}&\F{3}&\F{3}&\F{3}&\F{3}&\F{3}&\F{3}&\F{3}&\F{3}&\F{3}&\F{3}&\F{3}&\F{3}&{\cellcolor{compF!15}3}&{\cellcolor{compF!15}4}&\G{29}\\
{\cellcolor{compF!15}4}&\G{1}&{\cellcolor{compF!15}2}&{\cellcolor{compF!15}2}&{\cellcolor{compF!15}2}&{\cellcolor{compF!15}4}&\F{4}&\F{4}&\F{4}&\F{4}&\F{4}&\F{4}&\F{4}&\F{4}&\F{4}&\F{4}&\F{4}&\F{4}&\F{4}&\F{4}&\F{4}&\F{4}&\F{4}&\F{4}&\F{4}&\F{4}&\F{4}&{\cellcolor{compF!15}4}&{\cellcolor{compF!15}4}&\G{29}\\
{\cellcolor{compF!15}5}&\G{1}&{\cellcolor{compF!15}2}&{\cellcolor{compF!15}3}&{\cellcolor{compF!15}4}&{\cellcolor{compF!15}5}&\F{5}&\F{5}&\F{5}&\F{5}&\F{5}&\F{5}&\F{5}&\F{5}&\F{5}&\F{5}&\F{5}&\F{5}&\F{5}&\F{5}&\F{5}&\F{5}&\F{5}&\F{5}&\F{5}&\F{5}&\F{5}&{\cellcolor{compF!15}5}&{\cellcolor{compF!15}28}&\G{29}\\
{\cellcolor{compE!15}6}&\G{1}&\F{2}&\F{3}&\F{4}&\F{5}&{\cellcolor{compE!15}6}&{\cellcolor{compE!15}6}&{\cellcolor{compE!15}6}&\E{6}&\E{6}&\E{6}&\E{6}&\E{6}&\E{6}&\E{6}&\E{6}&\E{6}&\E{6}&\E{6}&\E{6}&\E{6}&\E{6}&\E{6}&{\cellcolor{compE!15}6}&{\cellcolor{compE!15}6}&{\cellcolor{compE!15}6}&\F{27}&\F{28}&\G{29}\\
{\cellcolor{compE!15}7}&\G{1}&\F{2}&\F{3}&\F{4}&\F{5}&{\cellcolor{compE!15}6}&{\cellcolor{compE!15}6}&{\cellcolor{compE!15}6}&\E{7}&\E{7}&\E{7}&\E{7}&\E{7}&\E{7}&\E{7}&\E{7}&\E{7}&\E{7}&\E{7}&\E{7}&\E{7}&\E{7}&\E{7}&{\cellcolor{compE!15}7}&{\cellcolor{compE!15}8}&{\cellcolor{compE!15}8}&\F{27}&\F{28}&\G{29}\\
{\cellcolor{compE!15}8}&\G{1}&\F{2}&\F{3}&\F{4}&\F{5}&{\cellcolor{compE!15}6}&{\cellcolor{compE!15}6}&{\cellcolor{compE!15}6}&\E{8}&\E{8}&\E{8}&\E{8}&\E{8}&\E{8}&\E{8}&\E{8}&\E{8}&\E{8}&\E{8}&\E{8}&\E{8}&\E{8}&\E{8}&{\cellcolor{compE!15}8}&{\cellcolor{compE!15}8}&{\cellcolor{compE!15}8}&\F{27}&\F{28}&\G{29}\\
{\cellcolor{compB!15}9}&\G{1}&\F{2}&\F{3}&\F{4}&\F{5}&\E{6}&\E{7}&\E{8}&{\cellcolor{compB!15}9}&{\cellcolor{compB!15}9}&\B{9}&\B{9}&\B{9}&\B{9}&\B{9}&{\cellcolor{compB!15}9}&{\cellcolor{compB!15}9}&{\cellcolor{compB!15}9}&{\cellcolor{compB!15}9}&\Cszin{20}&\D{21}&\D{22}&\D{23}&\E{24}&\E{25}&\E{26}&\F{27}&\F{28}&\G{29}\\
{\cellcolor{compB!15}10}&\G{1}&\F{2}&\F{3}&\F{4}&\F{5}&\E{6}&\E{7}&\E{8}&{\cellcolor{compB!15}9}&{\cellcolor{compB!15}10}&\B{10}&\B{10}&\B{10}&\B{10}&\B{10}&{\cellcolor{compB!15}10}&{\cellcolor{compB!15}17}&{\cellcolor{compB!15}17}&{\cellcolor{compB!15}19}&\Cszin{20}&\D{21}&\D{22}&\D{23}&\E{24}&\E{25}&\E{26}&\F{27}&\F{28}&\G{29}\\
{\cellcolor{compA!15}11}&\G{1}&\F{2}&\F{3}&\F{4}&\F{5}&\E{6}&\E{7}&\E{8}&\B{9}&\B{10}&{\cellcolor{compA!15}11}&{\cellcolor{compA!15}11}&{\cellcolor{compA!15}11}&{\cellcolor{compA!15}11}&{\cellcolor{compA!15}11}&\B{16}&\B{17}&\B{18}&\B{19}&\Cszin{20}&\D{21}&\D{22}&\D{23}&\E{24}&\E{25}&\E{26}&\F{27}&\F{28}&\G{29}\\
{\cellcolor{compA!15}12}&\G{1}&\F{2}&\F{3}&\F{4}&\F{5}&\E{6}&\E{7}&\E{8}&\B{9}&\B{10}&{\cellcolor{compA!15}11}&{\cellcolor{compA!15}11}&{\cellcolor{compA!15}11}&{\cellcolor{compA!15}12}&{\cellcolor{compA!15}13}&\B{16}&\B{17}&\B{18}&\B{19}&\Cszin{20}&\D{21}&\D{22}&\D{23}&\E{24}&\E{25}&\E{26}&\F{27}&\F{28}&\G{29}\\
{\cellcolor{compA!15}13}&\G{1}&\F{2}&\F{3}&\F{4}&\F{5}&\E{6}&\E{7}&\E{8}&\B{9}&\B{10}&{\cellcolor{compA!15}11}&{\cellcolor{compA!15}11}&{\cellcolor{compA!15}11}&{\cellcolor{compA!15}13}&{\cellcolor{compA!15}13}&\B{16}&\B{17}&\B{18}&\B{19}&\Cszin{20}&\D{21}&\D{22}&\D{23}&\E{24}&\E{25}&\E{26}&\F{27}&\F{28}&\G{29}\\
{\cellcolor{compA!15}14}&\G{1}&\F{2}&\F{3}&\F{4}&\F{5}&\E{6}&\E{7}&\E{8}&\B{9}&\B{10}&{\cellcolor{compA!15}11}&{\cellcolor{compA!15}12}&{\cellcolor{compA!15}13}&{\cellcolor{compA!15}14}&{\cellcolor{compA!15}15}&\B{16}&\B{17}&\B{18}&\B{19}&\Cszin{20}&\D{21}&\D{22}&\D{23}&\E{24}&\E{25}&\E{26}&\F{27}&\F{28}&\G{29}\\
{\cellcolor{compA!15}15}&\G{1}&\F{2}&\F{3}&\F{4}&\F{5}&\E{6}&\E{7}&\E{8}&\B{9}&\B{10}&{\cellcolor{compA!15}11}&{\cellcolor{compA!15}13}&{\cellcolor{compA!15}13}&{\cellcolor{compA!15}15}&{\cellcolor{compA!15}15}&\B{16}&\B{17}&\B{18}&\B{19}&\Cszin{20}&\D{21}&\D{22}&\D{23}&\E{24}&\E{25}&\E{26}&\F{27}&\F{28}&\G{29}\\
{\cellcolor{compB!15}16}&\G{1}&\F{2}&\F{3}&\F{4}&\F{5}&\E{6}&\E{7}&\E{8}&{\cellcolor{compB!15}9}&{\cellcolor{compB!15}10}&\B{16}&\B{16}&\B{16}&\B{16}&\B{16}&{\cellcolor{compB!15}16}&{\cellcolor{compB!15}17}&{\cellcolor{compB!15}18}&{\cellcolor{compB!15}19}&\Cszin{20}&\D{21}&\D{22}&\D{23}&\E{24}&\E{25}&\E{26}&\F{27}&\F{28}&\G{29}\\
{\cellcolor{compB!15}17}&\G{1}&\F{2}&\F{3}&\F{4}&\F{5}&\E{6}&\E{7}&\E{8}&{\cellcolor{compB!15}9}&{\cellcolor{compB!15}17}&\B{17}&\B{17}&\B{17}&\B{17}&\B{17}&{\cellcolor{compB!15}17}&{\cellcolor{compB!15}19}&{\cellcolor{compB!15}19}&{\cellcolor{compB!15}19}&\Cszin{20}&\D{21}&\D{22}&\D{23}&\E{24}&\E{25}&\E{26}&\F{27}&\F{28}&\G{29}\\
{\cellcolor{compB!15}18}&\G{1}&\F{2}&\F{3}&\F{4}&\F{5}&\E{6}&\E{7}&\E{8}&{\cellcolor{compB!15}9}&{\cellcolor{compB!15}18}&\B{18}&\B{18}&\B{18}&\B{18}&\B{18}&{\cellcolor{compB!15}18}&{\cellcolor{compB!15}19}&{\cellcolor{compB!15}19}&{\cellcolor{compB!15}19}&\Cszin{20}&\D{21}&\D{22}&\D{23}&\E{24}&\E{25}&\E{26}&\F{27}&\F{28}&\G{29}\\
{\cellcolor{compB!15}19}&\G{1}&\F{2}&\F{3}&\F{4}&\F{5}&\E{6}&\E{7}&\E{8}&{\cellcolor{compB!15}9}&{\cellcolor{compB!15}19}&\B{19}&\B{19}&\B{19}&\B{19}&\B{19}&{\cellcolor{compB!15}19}&{\cellcolor{compB!15}19}&{\cellcolor{compB!15}19}&{\cellcolor{compB!15}19}&\Cszin{20}&\D{21}&\D{22}&\D{23}&\E{24}&\E{25}&\E{26}&\F{27}&\F{28}&\G{29}\\
{\cellcolor{compC!15}20}&\G{1}&\F{2}&\F{3}&\F{4}&\F{5}&\E{6}&\E{7}&\E{8}&\Cszin{20}&\Cszin{20}&\Cszin{20}&\Cszin{20}&\Cszin{20}&\Cszin{20}&\Cszin{20}&\Cszin{20}&\Cszin{20}&\Cszin{20}&\Cszin{20}&{\cellcolor{compC!15}20}&\D{21}&\D{22}&\D{23}&\E{24}&\E{25}&\E{26}&\F{27}&\F{28}&\G{29}\\
{\cellcolor{compD!15}21}&\G{1}&\F{2}&\F{3}&\F{4}&\F{5}&\E{6}&\E{7}&\E{8}&\D{21}&\D{21}&\D{21}&\D{21}&\D{21}&\D{21}&\D{21}&\D{21}&\D{21}&\D{21}&\D{21}&\D{21}&{\cellcolor{compD!15}21}&{\cellcolor{compD!15}22}&{\cellcolor{compD!15}23}&\E{24}&\E{25}&\E{26}&\F{27}&\F{28}&\G{29}\\
{\cellcolor{compD!15}22}&\G{1}&\F{2}&\F{3}&\F{4}&\F{5}&\E{6}&\E{7}&\E{8}&\D{22}&\D{22}&\D{22}&\D{22}&\D{22}&\D{22}&\D{22}&\D{22}&\D{22}&\D{22}&\D{22}&\D{22}&{\cellcolor{compD!15}22}&{\cellcolor{compD!15}23}&{\cellcolor{compD!15}23}&\E{24}&\E{25}&\E{26}&\F{27}&\F{28}&\G{29}\\
{\cellcolor{compD!15}23}&\G{1}&\F{2}&\F{3}&\F{4}&\F{5}&\E{6}&\E{7}&\E{8}&\D{23}&\D{23}&\D{23}&\D{23}&\D{23}&\D{23}&\D{23}&\D{23}&\D{23}&\D{23}&\D{23}&\D{23}&{\cellcolor{compD!15}23}&{\cellcolor{compD!15}23}&{\cellcolor{compD!15}23}&\E{24}&\E{25}&\E{26}&\F{27}&\F{28}&\G{29}\\
{\cellcolor{compE!15}24}&\G{1}&\F{2}&\F{3}&\F{4}&\F{5}&{\cellcolor{compE!15}6}&{\cellcolor{compE!15}7}&{\cellcolor{compE!15}8}&\E{24}&\E{24}&\E{24}&\E{24}&\E{24}&\E{24}&\E{24}&\E{24}&\E{24}&\E{24}&\E{24}&\E{24}&\E{24}&\E{24}&\E{24}&{\cellcolor{compE!15}24}&{\cellcolor{compE!15}25}&{\cellcolor{compE!15}26}&\F{27}&\F{28}&\G{29}\\
{\cellcolor{compE!15}25}&\G{1}&\F{2}&\F{3}&\F{4}&\F{5}&{\cellcolor{compE!15}6}&{\cellcolor{compE!15}8}&{\cellcolor{compE!15}8}&\E{25}&\E{25}&\E{25}&\E{25}&\E{25}&\E{25}&\E{25}&\E{25}&\E{25}&\E{25}&\E{25}&\E{25}&\E{25}&\E{25}&\E{25}&{\cellcolor{compE!15}25}&{\cellcolor{compE!15}25}&{\cellcolor{compE!15}26}&\F{27}&\F{28}&\G{29}\\
{\cellcolor{compE!15}26}&\G{1}&\F{2}&\F{3}&\F{4}&\F{5}&{\cellcolor{compE!15}6}&{\cellcolor{compE!15}8}&{\cellcolor{compE!15}8}&\E{26}&\E{26}&\E{26}&\E{26}&\E{26}&\E{26}&\E{26}&\E{26}&\E{26}&\E{26}&\E{26}&\E{26}&\E{26}&\E{26}&\E{26}&{\cellcolor{compE!15}26}&{\cellcolor{compE!15}26}&{\cellcolor{compE!15}26}&\F{27}&\F{28}&\G{29}\\
{\cellcolor{compF!15}27}&\G{1}&{\cellcolor{compF!15}2}&{\cellcolor{compF!15}3}&{\cellcolor{compF!15}4}&{\cellcolor{compF!15}5}&\F{27}&\F{27}&\F{27}&\F{27}&\F{27}&\F{27}&\F{27}&\F{27}&\F{27}&\F{27}&\F{27}&\F{27}&\F{27}&\F{27}&\F{27}&\F{27}&\F{27}&\F{27}&\F{27}&\F{27}&\F{27}&{\cellcolor{compF!15}27}&{\cellcolor{compF!15}28}&\G{29}\\
{\cellcolor{compF!15}28}&\G{1}&{\cellcolor{compF!15}2}&{\cellcolor{compF!15}4}&{\cellcolor{compF!15}4}&{\cellcolor{compF!15}28}&\F{28}&\F{28}&\F{28}&\F{28}&\F{28}&\F{28}&\F{28}&\F{28}&\F{28}&\F{28}&\F{28}&\F{28}&\F{28}&\F{28}&\F{28}&\F{28}&\F{28}&\F{28}&\F{28}&\F{28}&\F{28}&{\cellcolor{compF!15}28}&{\cellcolor{compF!15}28}&\G{29}\\
{\cellcolor{compG!15}29}&{\cellcolor{compG!15}1}&\G{29}&\G{29}&\G{29}&\G{29}&\G{29}&\G{29}&\G{29}&\G{29}&\G{29}&\G{29}&\G{29}&\G{29}&\G{29}&\G{29}&\G{29}&\G{29}&\G{29}&\G{29}&\G{29}&\G{29}&\G{29}&\G{29}&\G{29}&\G{29}&\G{29}&\G{29}&\G{29}&{\cellcolor{compG!15}29}
\end{array}
\\[50mm]
\hskip1mm
\hspace*{0.6em}\begin{array}{c*{29}{c}}
&{\cellcolor{compG!15}29}&{\cellcolor{compF!15}28}&{\cellcolor{compF!15}27}&{\cellcolor{compF!15}28}&{\cellcolor{compF!15}27}&{\cellcolor{compE!15}26}&{\cellcolor{compE!15}24}&{\cellcolor{compE!15}26}&{\cellcolor{compB!15}19}&{\cellcolor{compB!15}16}&{\cellcolor{compA!15}15}&{\cellcolor{compA!15}14}&{\cellcolor{compA!15}15}&{\cellcolor{compA!15}14}&{\cellcolor{compA!15}15}&{\cellcolor{compB!15}16}&{\cellcolor{compB!15}16}&{\cellcolor{compB!15}16}&{\cellcolor{compB!15}19}&{\cellcolor{compC!15}20}&{\cellcolor{compD!15}21}&{\cellcolor{compD!15}21}&{\cellcolor{compD!15}23}&{\cellcolor{compE!15}24}&{\cellcolor{compE!15}25}&{\cellcolor{compE!15}26}&{\cellcolor{compF!15}27}&{\cellcolor{compF!15}28}&{\cellcolor{compG!15}29}
\end{array}
\end{array}
\]
}

\vspace{-1.8em}

{
\tiny
\setlength{\arraycolsep}{1.4pt}
\renewcommand{\arraystretch}{1.2}
\[
\begin{array}{c@{\qquad}c@{\qquad}c@{\qquad}c@{\qquad}c@{\qquad}c@{\qquad}c}
\left[
\begin{array}{ccccc}
{\cellcolor{compA!15}1}&{\cellcolor{compA!15}1}&{\cellcolor{compA!15}1}&{\cellcolor{compA!15}1}&{\cellcolor{compA!15}1}\\
{\cellcolor{compA!15}1}&{\cellcolor{compA!15}1}&{\cellcolor{compA!15}1}&{\cellcolor{compA!15}2}&{\cellcolor{compA!15}3}\\
{\cellcolor{compA!15}1}&{\cellcolor{compA!15}1}&{\cellcolor{compA!15}1}&{\cellcolor{compA!15}3}&{\cellcolor{compA!15}3}\\
{\cellcolor{compA!15}1}&{\cellcolor{compA!15}2}&{\cellcolor{compA!15}3}&{\cellcolor{compA!15}4}&{\cellcolor{compA!15}5}\\
{\cellcolor{compA!15}1}&{\cellcolor{compA!15}3}&{\cellcolor{compA!15}3}&{\cellcolor{compA!15}5}&{\cellcolor{compA!15}5}
\end{array}
\right]
&
\left[
\begin{array}{cccccc}
{\cellcolor{compB!15}1}&{\cellcolor{compB!15}1}&{\cellcolor{compB!15}1}&{\cellcolor{compB!15}1}&{\cellcolor{compB!15}1}&{\cellcolor{compB!15}1}\\
{\cellcolor{compB!15}1}&{\cellcolor{compB!15}2}&{\cellcolor{compB!15}2}&{\cellcolor{compB!15}4}&{\cellcolor{compB!15}4}&{\cellcolor{compB!15}6}\\
{\cellcolor{compB!15}1}&{\cellcolor{compB!15}2}&{\cellcolor{compB!15}3}&{\cellcolor{compB!15}4}&{\cellcolor{compB!15}5}&{\cellcolor{compB!15}6}\\
{\cellcolor{compB!15}1}&{\cellcolor{compB!15}4}&{\cellcolor{compB!15}4}&{\cellcolor{compB!15}6}&{\cellcolor{compB!15}6}&{\cellcolor{compB!15}6}\\
{\cellcolor{compB!15}1}&{\cellcolor{compB!15}5}&{\cellcolor{compB!15}5}&{\cellcolor{compB!15}6}&{\cellcolor{compB!15}6}&{\cellcolor{compB!15}6}\\
{\cellcolor{compB!15}1}&{\cellcolor{compB!15}6}&{\cellcolor{compB!15}6}&{\cellcolor{compB!15}6}&{\cellcolor{compB!15}6}&{\cellcolor{compB!15}6}
\end{array}
\right]
&
\left[
\begin{array}{c}
{\cellcolor{compC!15}1}
\end{array}
\right]
&
\left[
\begin{array}{ccc}
{\cellcolor{compD!15}1}&{\cellcolor{compD!15}2}&{\cellcolor{compD!15}3}\\
{\cellcolor{compD!15}2}&{\cellcolor{compD!15}3}&{\cellcolor{compD!15}3}\\
{\cellcolor{compD!15}3}&{\cellcolor{compD!15}3}&{\cellcolor{compD!15}3}
\end{array}
\right]
&
\left[
\begin{array}{cccccc}
{\cellcolor{compE!15}1}&{\cellcolor{compE!15}1}&{\cellcolor{compE!15}1}&{\cellcolor{compE!15}1}&{\cellcolor{compE!15}1}&{\cellcolor{compE!15}1}\\
{\cellcolor{compE!15}1}&{\cellcolor{compE!15}1}&{\cellcolor{compE!15}1}&{\cellcolor{compE!15}2}&{\cellcolor{compE!15}3}&{\cellcolor{compE!15}3}\\
{\cellcolor{compE!15}1}&{\cellcolor{compE!15}1}&{\cellcolor{compE!15}1}&{\cellcolor{compE!15}3}&{\cellcolor{compE!15}3}&{\cellcolor{compE!15}3}\\
{\cellcolor{compE!15}1}&{\cellcolor{compE!15}2}&{\cellcolor{compE!15}3}&{\cellcolor{compE!15}4}&{\cellcolor{compE!15}5}&{\cellcolor{compE!15}6}\\
{\cellcolor{compE!15}1}&{\cellcolor{compE!15}3}&{\cellcolor{compE!15}3}&{\cellcolor{compE!15}5}&{\cellcolor{compE!15}5}&{\cellcolor{compE!15}6}\\
{\cellcolor{compE!15}1}&{\cellcolor{compE!15}3}&{\cellcolor{compE!15}3}&{\cellcolor{compE!15}6}&{\cellcolor{compE!15}6}&{\cellcolor{compE!15}6}
\end{array}
\right]
&
\left[
\begin{array}{cccccc}
{\cellcolor{compF!15}1}&{\cellcolor{compF!15}1}&{\cellcolor{compF!15}1}&{\cellcolor{compF!15}1}&{\cellcolor{compF!15}1}&{\cellcolor{compF!15}1}\\
{\cellcolor{compF!15}1}&{\cellcolor{compF!15}1}&{\cellcolor{compF!15}1}&{\cellcolor{compF!15}2}&{\cellcolor{compF!15}2}&{\cellcolor{compF!15}3}\\
{\cellcolor{compF!15}1}&{\cellcolor{compF!15}1}&{\cellcolor{compF!15}1}&{\cellcolor{compF!15}3}&{\cellcolor{compF!15}3}&{\cellcolor{compF!15}3}\\
{\cellcolor{compF!15}1}&{\cellcolor{compF!15}2}&{\cellcolor{compF!15}3}&{\cellcolor{compF!15}4}&{\cellcolor{compF!15}4}&{\cellcolor{compF!15}6}\\
{\cellcolor{compF!15}1}&{\cellcolor{compF!15}2}&{\cellcolor{compF!15}3}&{\cellcolor{compF!15}4}&{\cellcolor{compF!15}5}&{\cellcolor{compF!15}6}\\
{\cellcolor{compF!15}1}&{\cellcolor{compF!15}3}&{\cellcolor{compF!15}3}&{\cellcolor{compF!15}6}&{\cellcolor{compF!15}6}&{\cellcolor{compF!15}6}
\end{array}
\right]
&
\left[
\begin{array}{cc}
{\cellcolor{compG!15}1}&{\cellcolor{compG!15}1}\\
{\cellcolor{compG!15}1}&{\cellcolor{compG!15}2}
\end{array}
\right]
\end{array}
\]
}
}
\caption{An example of the canonical $\tau$-multiplication-coherent decomposition.
The fifth component monoid is $\tau$-multiplication-cohesive but not $\tau$-cohesive; the sixth is $\tau$-cohesive; and the last is integral, hence $\tau$-cohesive.
The ambient monoid is residuated, but its fourth component is not.
The cells without background coloring show the rigidity phenomenon.}
\label{fig:Example} 
\end{figure}

\begin{remark}
Figure~\ref{fig:Example} illustrates a concrete decomposition of a monoid: the large multiplication table is shown together with its \(\tau\)-row, and the corresponding component monoids are displayed below.
For readability, the small multiplication tables are displayed on the standard carriers $\{1,\dots,n\}$; hence they represent monoids isomorphic to the actual components, rather than the components on their original decomposition blocks.
\end{remark}

\begin{remark}\label{rem:not-res} 
Although the idea of stratifying a structure via the local-unit map \(\tau\) was first introduced in a residuated setting in \cite{Jenei2022GroupRepr}, that framework is not the natural ambient one for the present theory; see \cite{GalatosJipsenKowalskiOno2007} for a comprehensive account of residuated lattices.
For the present theory, it would already be enough for the component monoids to be residuated, even if their residual operations were not inherited from the original structure.
Yet Figure~\ref{fig:Example} shows that even this weaker requirement can fail: although the input monoid is residuated and admits a nontrivial decomposition governed by \(\tau\), its fourth component monoid is not residuated\footnote{For finite totally ordered monoids, residuatedness is equivalent to the least element being a two-sided zero.
Indeed, if the least element \(\bot\) is a two-sided zero, then all residual solution sets are nonempty and hence have maxima.
Conversely, if residuals exist, then $x\backslash\bot\ge\bot$ and $\bot/x\ge\bot$, hence $x\bot\le\bot$ and $\bot x\le\bot$ by adjointness, yielding that $\bot$ is a two-sided zero.}.

Hence the decomposition mechanism is essentially monoid-theoretic rather than residuation-theoretic, and in the residuated setting further assumptions are needed to guarantee that residuation is preserved under decomposition.
\end{remark}

\begin{remark}
The rigidity result for canonical transition maps brings the present theory close in spirit to Clifford's ordinal-sum theory for naturally totally ordered commutative semigroups \cite{Clifford1954,Clifford1958}.
In Clifford's setting, the components are linearly arranged, and the upper component in a reducible step is absorbent; thus mixed products are governed entirely by the higher component.
The present paper recovers an analogous phenomenon in direct-system language: if \( A<_{\min}B, \) then the canonical transition map \( \rho_{A,B}\colon M_A\to M_B \) is unit-constant, and hence for all $x\in M_A$, $y\in M_B$, \( xy=y=yx. \)
Thus the lower component does not produce a genuinely new mixed layer.

The analogy should not be read as a literal extension of Clifford's theorem.
The present framework is more general in that commutativity is not assumed, the order need not be natural in Clifford's sense, and the identity need not lie at either endpoint of the chain.
At the same time, it is more restrictive in that we work only with finite monoids.
The representation theorem established here is therefore best viewed as a Clifford-type analogue in a different direction: finite, monoidal, noncommutative, and governed by local-unit geometry rather than natural order.
\end{remark}

\section*{Conclusion}

We have shown that finite local-unit-aligned totally ordered monoids admit a canonical internal stratification induced by the local-unit map $\tau$, and that this stratification is exactly the structure needed for decomposition and reconstruction.
The canonical $\tau$-multiplication-coherent decomposition yields component monoids and transition maps forming a rigid finite chain-indexed direct system; from this system both the ambient order and the multiplication can be recovered.
Conversely, every rigid system of this kind canonically reconstructs a finite local-unit-aligned totally ordered monoid.

This places the finite local-unit-aligned setting within a Clifford-type ordinal-sum-like representation framework.
Before rigidity is used, the decomposition is Clifford-like in the strong-semilattice sense: the monoid is recovered from intrinsic components and canonical transition maps.
Once finiteness forces rigidity, those transition maps become unit-constant, and the same decomposition becomes Clifford-like in the ordinal-sum sense: cross-component multiplication is governed by absorption toward the upper component.
Thus the theory simultaneously extends the component-and-map reconstruction principle and explains why, in the finite ordered setting, that principle collapses to an ordinal-sum-like form.

The main contribution of the paper is therefore the identification of the intrinsic local-unit decomposition through which the global structure becomes both reconstructible and rigid.
The canonical \(\tau\)-multiplication-coherent quotient is not only a stratification of the positive idempotent skeleton: it is the quotient from which the component monoids, the transition maps, and the cross-component multiplication are determined.
In the finite setting, strict compatibility then forces these transition maps to be unit-constant, explaining why the resulting representation has an ordinal-sum-like form.

Extensions beyond the local-unit-aligned hypothesis, to infinite monoids, and to further structural refinements of the representation involve substantially different arguments and will be treated separately.

\backmatter

\bmhead*{Acknowledgements}

The author gratefully acknowledges support from the Ministry of Culture and Innovation of Hungary, through the National Research, Development and Innovation Fund, grant no.~K138596.
\bigskip

\section*{Declarations}

\noindent\textbf{Competing interests.}
The author declares that he has no competing interests.

\noindent\textbf{Data availability.}
No external datasets were used.
The finite examples and minimality assertions mentioned in the paper are based on finite exhaustive enumeration; the corresponding verification is elementary and can be reproduced directly from the multiplication tables and the defining conditions.

\end{document}